\documentclass[a4paper,11pt,reqno]{amsart}

\usepackage{amssymb,latexsym,amsthm,amsmath,amsfonts}
\usepackage{caption,color,graphicx}
\usepackage[numbers,comma,square,sort&compress]{natbib}
\usepackage{hyperref} 
\usepackage[text={5.5in,9.5in},centering]{geometry}
\usepackage{bbm}
\usepackage{enumitem}
\usepackage{mdframed}

\usepackage{mathtools}

\setcaptionmargin{0.25in}

\newtheorem{lemma}{Lemma}[section]

\newtheorem{definition}{Definition}[section]

\newtheorem{proposition}{Proposition}[section]
\newtheorem{remark}{Remark}[section]
\newtheorem{theorem}{Theorem}[section]
\newtheorem{example}{Example}[section]

\newcommand{\I}{\mathbb{I}}

\title[On-off intermittency and chaotic walks]{On-off intermittency and chaotic walks}
 \author{Ale Jan Homburg}
 \address{KdV Institute for Mathematics, University of Amsterdam, Science park 107, 1098 XG Amsterdam, Netherlands}
 \address{Department of Mathematics, VU University Amsterdam, De Boelelaan 1081, 1081 HV Amsterdam, Netherlands}
 \email{a.j.homburg@uva.nl}
 \author{Vahatra Rabodonandrianandraina}
 \address{KdV Institute for Mathematics, University of Amsterdam, Science park 107, 1098 XG Amsterdam, Netherlands}
 \email{v.f.rabodonandrianandraina@uva.nl}

\begin{document}

\begin{abstract}
We consider a class of skew product maps of interval diffeomorphisms over the doubling map. The interval maps fix the end points of the interval.
It is assumed that the system has zero fiber Lyapunov exponent at one endpoint and zero or positive fiber Lyapunov exponent at the other endpoint. 
We  prove the appearance of on-off intermittency.
This is done using the equivalent description of 
chaotic walks: random walks driven by the doubling map. The analysis further relies on approximating the chaotic walks by Markov random walks, 
that are constructed using Markov partitions for the doubling map.
\end{abstract}

\maketitle

\section{Introduction}

The setting of this paper is
of skew product systems of interval maps over 
linearly expanding interval maps
\begin{align*}
\hat{G}(y,x) &=  (E_m (y),   \hat{g}_y (x))
\end{align*}
on $\I \times [0,1]$.
Here $E_m : \I \to \I$, $m$ an integer bigger or equal than $2$, is the expanding map
\begin{align*}
E_m (y) &= my - \lfloor my \rfloor, 
\end{align*}
and we use notation $\I = [0,1]$ for the base space on which $E_m$ acts.
For each $y \in \I$, $\hat{g}_y$ is a strictly monotone interval map fixing the endpoints $\hat{g}_y(0)=0$ and $\hat{g}_y(1) =1$.
It is assumed to be smooth jointly in $(y,x)$.
We will in particular consider the doubling map $E_2$ and we will limit to this now and return to stronger expanding maps later. 



Since each $\hat{g}_y$ fixes the endpoints of the interval $[0,1]$,
we can conjugate the map $\hat{G}$ on $\I \times (0,1)$ to a map $G$ on $\I \times \mathbb{R}$, using the homeomorphism
$h: \mathbb{R} \to (0,1)$,
\begin{align*}
h(x) &= \frac{e^x}{1 + e^x}.
\end{align*}
That is,
\begin{align*}
G(y,x) &= (E_2(y) , g_y (x)) =  (E_2(y),  h^{-1} \circ \hat{g}_y \circ h  (x)).
\end{align*}
Observe that $h^{-1}(x) = \ln (x/1-x)$. A small calculation shows that near $-\infty$ we can write
\begin{align*}
 g_y (x) &= x + \ln (\hat{g}_y'(0)) + R(y,x) 
\end{align*}
with
\begin{align*}
|R(y,x)| &\le C e^{-|x|}
\end{align*}
for some constant $C>0$. A similar expansion applies near $+\infty$.

\begin{figure}[phtb]

\begin{center}
\includegraphics[width=12cm]{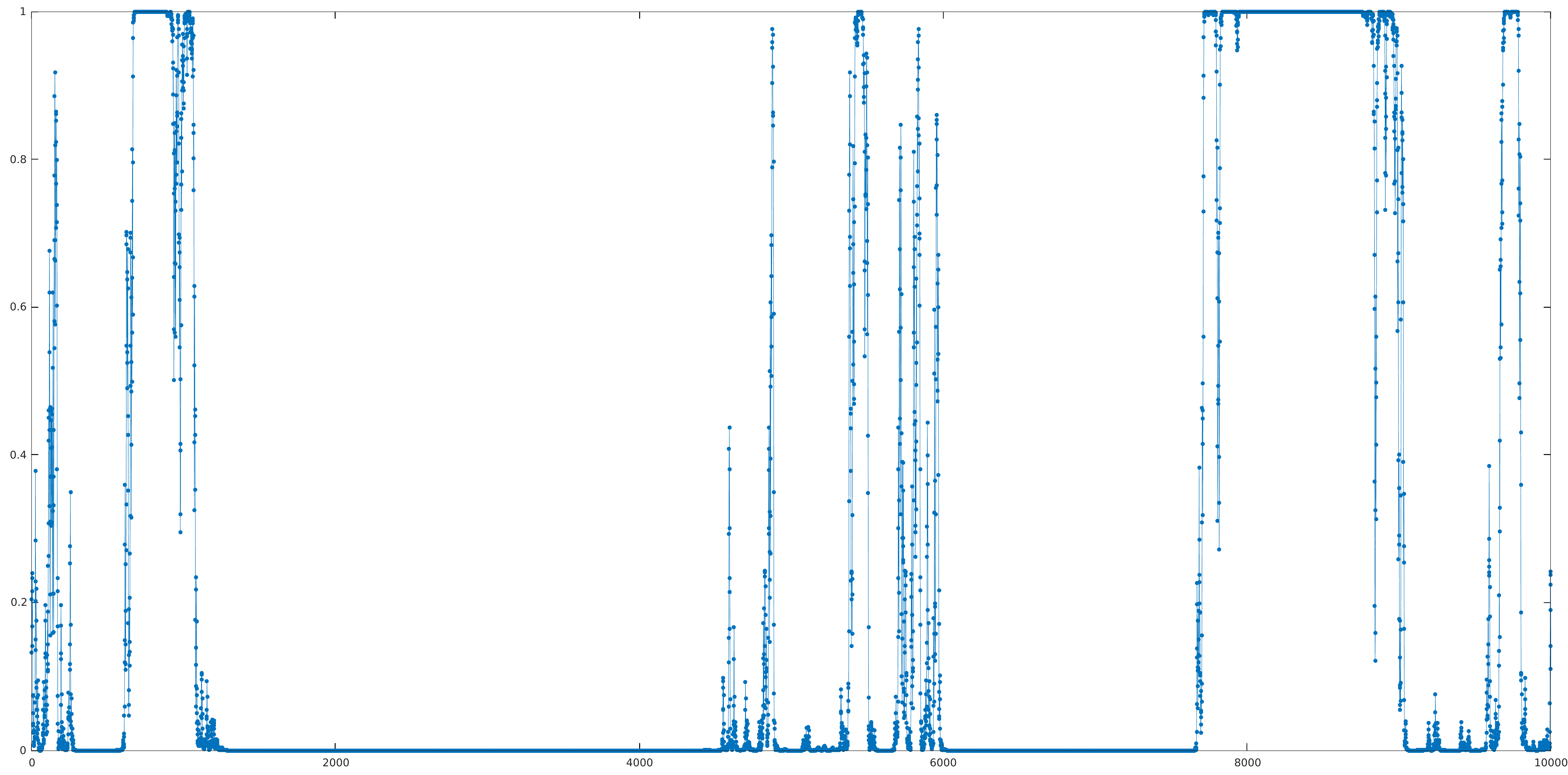} 
\caption{Time series of the $x$-coordinate of 
$
(y,x) \mapsto \left(3 y \mod 1 , \frac{x e^{-1+2y}}{  1 + x(e^{-1+2y}-1)}\right)
$ 
on $\I \times [0,1]$. The multiplication by $3$, instead of $2$, is for computational convenience.
The restriction of this map to $\I \times (0,1)$ is topologically conjugate to the group extension
$
(y,x) \mapsto (3y \mod 1 , x -1 + 2y)
$ 
on $\I \times \mathbb{R}$. For almost all initial values $y$, the distribution of the steps is uniform on $[-1,1]$ and has average zero.
\label{f:-1+2y}}
\end{center}

\end{figure}
Particular examples are given by translations on the real line driven by the doubling map,
\begin{align}\label{e:walk}
(y,x) \mapsto (E_2 (y) , x + \xi (y)).
\end{align}
Iterates of the fiber coordinate yield the cocycle 
\[
x \mapsto x + \sum_{i=0}^{n-1} \xi (E_2^i(y)).
\]
By \cite{atk76}  
 this cocycle is recurrent precisely if 
 \begin{align}\label{e:xiaverage0}
 \int_{\I} \xi (y) \, dy = 0.
 \end{align}
See \cite{boygor97} for the central limit theorem for cocycles over intervals maps such as the doubling map.
References \citep{nit00,nitpol05,gui89,fiemeltor05} contain further results on ergodicity and stable transitivity for similar cocycles.

Conjugating \eqref{e:walk} back to $\I \times (0,1)$ results in the skew product system
\begin{align}\label{e:hatG}
(y,x) \mapsto \left(E_2 (y), \frac{e^{\xi(y)} x}{1+(e^{\xi(y)} - 1)x}\right).
\end{align}
Figure~\ref{f:-1+2y} shows a time series for the $x$-component of  \eqref{e:hatG} for the choice $\xi(y) = -1+2y$
(for which \eqref{e:xiaverage0} applies).
Noticeable are the long durations of the time series near $x=0$ or $x=1$ and the bursts away.
From this perspective such skew product systems are considered in \cite[Section~6]{bonmil08}, but for statements on dynamics that paper replaces the action of $E_2$ 
by i.i.d. noise from a uniform distribution on $\I$. 

A broader framework for which the study of systems such as \eqref{e:hatG} is relevant is that of on-off intermittency \cite{plaspitre93}. 
On-off intermittency is associated to invariant manifolds with inside an attractor that is neutral or weakly repelling in transverse directions.
This can generate dynamics showing an aperiodic switching between laminar dynamics near the attractor and bursts away.
In \eqref{e:hatG}, $\I \times \{0\}$ and $\I\times \{1\}$  play the role of the invariant manifolds with chaotic dynamics inside and a neutral transverse direction. 
A different type of system, 
but also with vanishing transverse Lyapunov exponents, is considered in \cite{gou07}; that paper considers 
Pomeau-Manneville maps with a parameter that is driven by the doubling map.
In bifurcation studies the relevant transition is called a blowout bifurcation, where a transverse Lyapunov exponent passes through zero \citep{ottsom94,ashastnic98}.

Quantitative characteristics of on-off intermittency are considered in \cite{heaplaham94}.  Referring to systems such as \eqref{e:walk} as chaotic walks, the authors comment
\vspace{0.5cm}
\begin{mdframed}[innerleftmargin=15,innerrightmargin=15,hidealllines=true]
{\sloppy
Is there a setting in which to understand both the random 
and the chaotic driving cases?
We suggest that one approach to answering this question is through the study of "chaotic walks", i.e., 
additive walks where the increments are chosen from some chaotic process. [...]
We know of no systematic studies of chaotic
walks, though they are clearly an important counterpart
to the comparatively well studied random walks.}
%
\end{mdframed}
\vspace{0.5cm}
Random walks in $\mathbb{Z}$ driven by expanding Markov maps are considered in \cite{morsma14}, where they are called deterministic random walks.
Other authors have considered deterministic random walks driven by irrational circle rotations instead of chaotic maps such as the doubling map, 
see in particular \citep{aarkea82,avidoldursar15}.
In \cite{ghahom17} the reader can find a study of iterated function systems of interval diffeomorphisms with a neutral fixed point. 
This amounts to a study of 
chaotic walks on the line:  for instance the symmetric random walk  can be cast  as a skew product setting of the form
\[
 (y,x) \mapsto \left(E_2 (y) , x  + \mathrm{sign}\, \left(y-\frac{1}{2}\right)\right). 
\]
See Appendix~\ref{a:chaoticwalk} for comments on the relation and difference between chaotic walks and random walks.
In the same vain as \cite{ghahom17}, \citep{athdai00,athsch03} analyze iterated function systems of logistic maps with vanishing fiber Lyapunov exponent at zero. 

\subsection{The class of skew product systems}

We go beyond group extensions as in \eqref{e:walk} or \eqref{e:hatG}, and will consider small smooth perturbations of the form 
\[
\hat{G}(y,x) = (E_2(y) , \hat{g}_y (x)) = \left(E_2 (y) , \frac{e^{\xi(y)} x}{1+(e^{\xi(y)} - 1)x} + r(y,x) \right),
\]
where
$r (y,x) = \mathcal{O} (x^2)$, $x\to 0$. 
 

\begin{definition}\label{d:hatSone}
The set $\hat{\mathcal{S}}$ consists of smooth skew product systems $\hat{G}: \I \times [0,1] \to  \I \times [0,1]$,
\[
\hat{G} (y,x) = ( E_2 (y) , \hat{g}_y(x)).
\]
Here  $\hat{g}_y$ are strictly increasing functions,  
\[
 \hat{g}_y (x) = \frac{e^{\xi(y)} x}{1+(e^{\xi(y)} - 1)x} + r (y,x)
\]
with, for some $C>0, r_0>0$, 
\begin{enumerate}
  \item   $\int_{\I}  \xi(y)  dy =0$;
 \item $|r(y,x)| \le C x^2$,  $|r(y,x)| \le C |1-x|$;
 \item $|r(y,x)|, |\frac{\partial}{\partial x} r(y,x)|, |\frac{\partial}{\partial y} r(y,x)| \le r_0$. 
\end{enumerate}
\end{definition}

Note that $\hat{\mathcal{S}}$ depends on $\xi$ and $C,r_0$.  
We will occasionally write $\hat{\mathcal{S}}_{r_0}$ to indicate dependence on $r_0$, suppressing dependence on $C,\xi$. 
We write $\hat{G}^n (y,x) = (E_2^n(y) , \hat{g}_y^n (x))$, so that
\[
\hat{g}^n_y(x) = \hat{g}_{E_2^{n-1} (y)} \circ \cdots \circ \hat{g}_y (x).
\]

\begin{remark}
Equation~\eqref{e:xiaverage0}, noting that in addition $r(y,x) = \mathcal{O}(x^2)$ for $x\to 0$,  
expresses a vanishing fiber Lyapunov exponent at $0$:
\[
 L_0 := \int_{\I} \ln \hat{g}'_y (0) \, dy =  \int_{\I}  \xi(y) \, dy =  0.
\] 
The fiber Lyapunov exponent at $1$,
\begin{align}\label{e:hatRlyap}
 L_1 := \int_{\I}   \ln \hat{g}'_y(1)  \, dy
\end{align}  
need not vanish for maps $\hat{G} \in \hat{\mathcal{S}}$.
%
Clearly $L_1$ is small if $r_0$ is small.
 \end{remark}
  
In the setting of skew product systems on $\I \times \mathbb{R}$, Definition~\ref{d:hatSone} leads to the class of systems
\begin{align*}
 \mathcal{S} &=  \mathcal{S}_{r_0} = \{ G: \I \times \mathbb{R} \to \I \times \mathbb{R} \; ; \; G = (\mathrm{id}\times h^{-1}) \circ \hat{G} \circ (\mathrm{id} \times h), \hat{G} \in \hat{\mathcal{S}} \}. 
\end{align*}
The set $\mathcal{S}$ consists of smooth skew product systems $G: \I \times \mathbb{R} \to  \I \times \mathbb{R}$,
\begin{align}\label{e:GfromS}
G (y,x) = ( E_2 (y) , g_y(x)).
\end{align}
For these systems,  $g_y$ are strictly increasing functions of the form
\begin{align}\label{e:gy}
g_y(x) &= x + \xi (y) + R(y,x)
\end{align}
with
\begin{align*}
|R(y,x)|, |DR(y,x)| &\le C r_0,
\\
|R(y,x)| &\le C e^{x},
\end{align*} 
for some $C>0$.
Note that the perturbation $R(y,x)$ is exponentially flat at $- \infty$.  
One can view such a map as giving a nonhomogeneous chaotic walk.

\subsection{Monotone displacement functions} 

Here we treat displacement functions $\xi: \I \to \mathbb{R}$ that are strictly monotone functions satisfying \eqref{e:xiaverage0}.
Below we consider general smooth displacement functions.

Figure~\ref{f:-1+2ypert} shows a time series of a perturbation of the map considered
in Figure~\ref{f:-1+2y}, where the perturbation is chosen to keep a zero fiber Lyapunov exponent at $0$ and to 
get a positive fiber Lyapunov exponent at $1$.
\begin{figure}[phtb]

\begin{center}
\includegraphics[width=12cm]{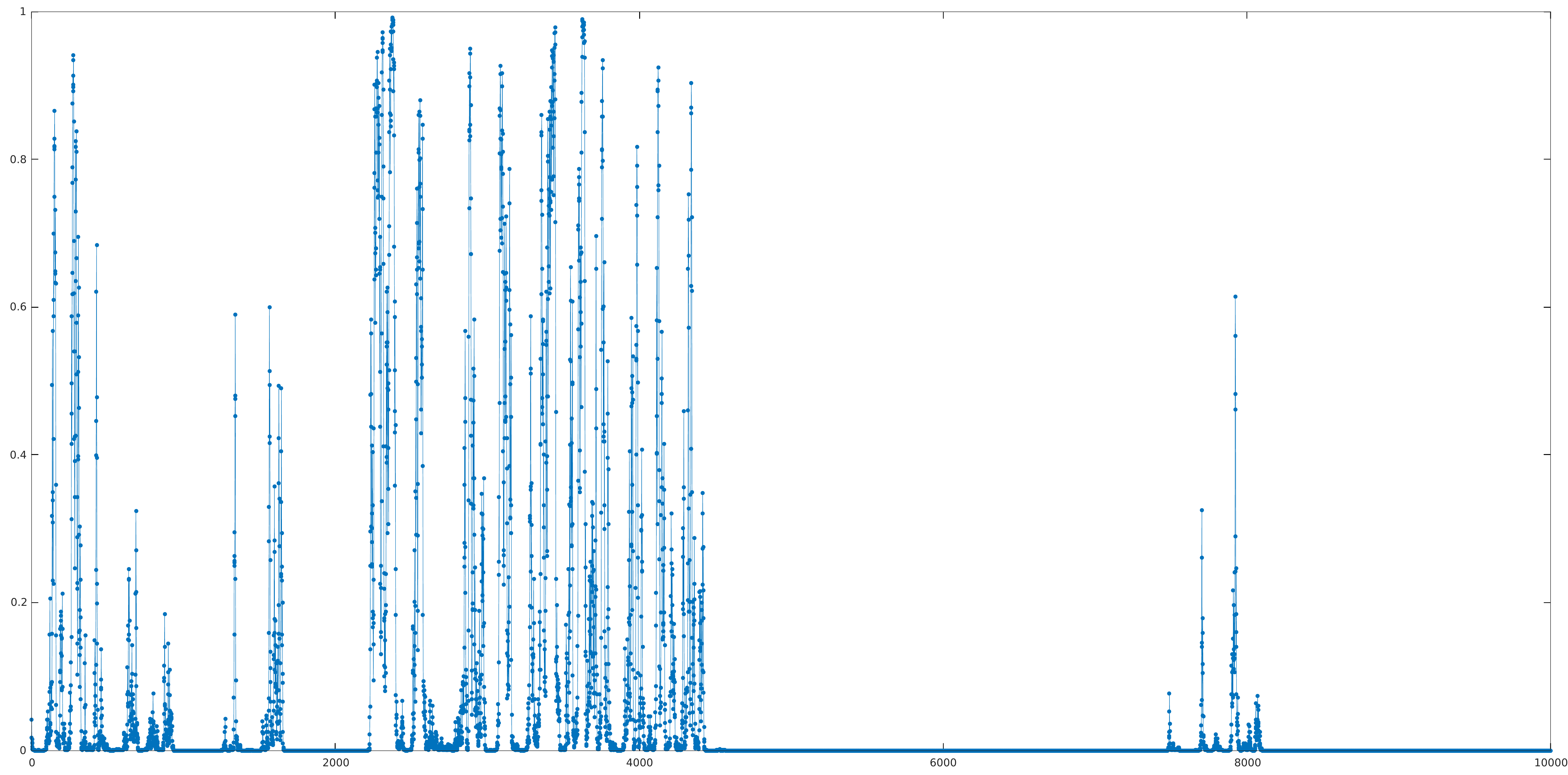} 
\caption{This plot features a perturbation of the skew product map considered in Figure~\ref{f:-1+2y}. 
Shown is a time series of the $x$-coordinate of 
$
(y,x) \mapsto \left(3 y \mod 1 , \frac{x e^{-1+2y}}{  1 + x(e^{-1+2y}-1)}  + \frac{1}{5} x^2(x-1) \right)
$ 
on $\I \times [0,1]$. 
\label{f:-1+2ypert}}
\end{center}

\end{figure}
As this makes $\I\times \{1\}$ repelling on average, we find the phenomenon of on-off intermittency only near $\I \times \{0\}$.

We formulate two theorems expressing aspects of on-off intermittency. The results are formulated in terms of the skew product maps
$\hat{G}$ on $\I\times [0,1]$.
The results can of course be phrased for the corresponding skew product systems $G$ 
on $\I \times \mathbb{R}$, that is in terms of chaotic walks on $\mathbb{R}$. 

For $x,p \in (0,1)$ with $x<p$, define
\[
T(y) = \min \{ n > 0 \; ; \; \hat{g}^n_y(x) > p \}.
\]
The following result states that $T$ has finite values almost everywhere,
but the average of the escape time $T$ is infinite.

\begin{theorem}\label{t:expisinf}
Let $\xi: \I \to \mathbb{R}$ be a smooth strictly monotone map satisfying \eqref{e:xiaverage0}.
For $r_0$ small enough, the following holds for $\hat{G} \in \hat{\mathcal{S}}_{r_0}$.
\begin{enumerate}
\item
$T(y) < \infty$ for Lebesgue almost all $y \in \I$;
\item
$
\int_{\I}  T(y) \, dy = \infty.
$
\end{enumerate}
\end{theorem}

For a subset $U$ of $\I$ write $\mathbbm{1}_U$ for the characteristic function of $U$.

\begin{theorem}\label{t:birkhoff00}
Let $\xi: \I \to \mathbb{R}$ be a smooth strictly monotone map satisfying \eqref{e:xiaverage0}. 
For $r_0$ small enough, the following holds for $\hat{G} \in \hat{\mathcal{S}}_{r_0}$.

Assume $L_0 = 0$ and $L_1>0$.
Let $U$ be a compact interval $[\epsilon,1] \subset (0,1]$ with $\epsilon>0$ small. 
Take $x \in (0,1)$.
Then  for Lebesgue almost all $y \in \I$,
the number of iterates $\hat{g}^i_y (x)$ in $U$ is infinite, but
\[
\lim_{n\to \infty} \frac{1}{n} \sum_{i=0}^{n-1} \mathbbm{1}_U ( \hat{g}^i_y (x)  ) = 0.
\]
\end{theorem}

The results imply that the only invariant probability measures for a skew product $\hat{G} \in \hat{\mathcal{S}}$
with Lebesgue measure as marginal, are the product measures of Lebesgue measure on $\I$ and convex combinations of delta measures at $0$ and $1$.

A similar result can be formulated for a skew product $\hat{G} \in \hat{\mathcal{S}}$  with vanishing fiber Lyapunov exponents
at both boundaries, i.e. where $L_1=0$ in \eqref{e:hatRlyap}: just replace $U$ by an interval $[\epsilon,1-\epsilon]$ with $\epsilon>0$ small. 

\begin{theorem}\label{t:birkhoff00two}
Let $\xi: \I \to \mathbb{R}$ be a smooth strictly monotone map satisfying \eqref{e:xiaverage0}. 
For $r_0$ small enough, the following holds for $\hat{G} \in \hat{\mathcal{S}}_{r_0}$.

Assume that both $L_0 = L_1 = 0$.
Let $U$ be a compact interval $[\epsilon,1-\epsilon] \subset (0,1)$ with $\epsilon>0$ small. 
Take $x \in (0,1)$.
Then  for Lebesgue almost all $y \in \I$,
the number of iterates $\hat{g}^i_y (x)$ in $U$ is infinite, but
\[
\lim_{n\to \infty} \frac{1}{n} \sum_{i=0}^{n-1} \mathbbm{1}_U ( \hat{g}^i_y (x)  ) = 0.
\]
\end{theorem}

\subsection{General displacement functions}

The results stated above are all for monotone displacement functions $\xi$.
The results are proved to hold for general smooth displacement functions $\xi$, if we replace the doubling map by an expanding map
$E_m (x) = m x \mod 1$ for large enough $m$.
We will formulate the result analogous to Theorem~\ref{t:birkhoff00}.
The class of skew product systems $\hat{G}(y,x) = (E_m (y) , \hat{g}_y (x))$ with fiber maps $\hat{g}_y$ as in Definition~\ref{d:hatSone},
is denoted by $\hat{\mathcal{S}}^m_{r_0}$.
The corresponding class of skew product systems $G(y,x) = (E_m (y) , g_y (x))$ is denoted by $\mathcal{S}^m_{r_0}$.

%
%

\begin{theorem}\label{t:birkhoffT}
Let $\xi: \I \to \mathbb{R}$ be a smooth map, not identically zero, satisfying \eqref{e:xiaverage0}. 

For $m \in \mathbb{N}$ large enough and $r_0>0$ small enough, the following holds for $\hat{G} \in \hat{\mathcal{S}}^{m}_{r_0}$.
Assume $L_0=0$ and $L_1 >0$.
Let $U$ be a compact interval $[\epsilon,1] \subset (0,1]$ with $\epsilon$ small. 
Take $x \in (0,1)$.
Then  for Lebesgue almost all $y \in \I$,
the number of iterates $\hat{g}^i_y (x)$ in $U$ is infinite, but
\[
\lim_{n\to \infty} \frac{1}{n} \sum_{i=0}^{n-1} \mathbbm{1}_U ( \hat{g}^i_y (x)  ) = 0.
\]
\end{theorem}


Above we considered interval diffeomorphisms forced by expanding maps on the interval $\I$.
Another natural context is that of interval diffeomorphisms forced by the (linearly) expanding circle maps.
Here we are given a smooth function $\xi:\mathbb{T} \to \mathbb{R}$ (with $\mathbb{T} = \mathbb{R} / \mathbb{Z}$).
This is included in Theorem~\ref{t:birkhoffT}, by taking functions $\xi: \I \to \mathbb{R}$ that give smooth functions on the circle 
$\mathbb{T} = \mathbb{R} / \mathbb{Z}$ when identifying $0$ and $1$.   
So Theorem~\ref{t:birkhoffT} contains as a special case a statement on skew products on $\mathbb{T} \times [0,1]$
that are forced by sufficiently expanding circle maps.

\subsection{Methodology}

We finish the introduction with a brief account of the approach we take in this paper. 
The reasoning will be in the setting of skew products with the real line as fiber.
We start with a chaotic walk
\[
 x_{n+1} = x_n + \xi (y_n) + R(y_n,x_n),
\]
where $y_{n+1} = E_2 (y_n)$ and $R$ is small, so that we have a nonhomogeneous walk close to a homogeneous walk.
More specific, as the chaotic walk originates from a smooth skew product system on $\I \times [0,1]$, we have that $R(y,x)$ is exponentially small in  $x$ for 
$x$ near $-\infty$.

The chaotic walk driven by the doubling map is rewritten as a chaotic walk driven by a subshift of finite type.
The subshift is obtained from a Markov partition for $E_2$. We will consider increasingly fine Markov partitions indexed by an integer $N$.
This yields subshifts $\sigma$ on sequence spaces $\Sigma_{\mathcal{A}_N}$ with an increasing number of symbols. 
Using a measurable isomorphism of $E_2$ with $\sigma$,
the system is written as a chaotic walk
\[
 x_{n+1} = x_n + \xi (\sigma^n \omega) + R (\sigma^n \omega,x_n).
\]
This chaotic walk is approximated by Markov random walks driven by subshifts, i.e. by walks of the form
\[
 x_{n+1} = x_n + \xi_N (\omega_n) + R_N(\omega_n,x_n)
\]
where $\omega = (\omega_i)_{i\in\mathbb{N}}$. 
Specifically, the maps $\xi_N$ and $R_N$ are constant in $\omega$ on cylinders of rank one. 
The approximations get better for increasing $N$ when the partition elements get smaller diameter.

Our main results are proven using results on escape times from subintervals of the line for the approximate Markov random walks,
and carefully taking limits of increasingly fine Markov partitions. 
The derivation of the results on Markov random walks is contained in the appendices.

We expect that various generalizations of the results in this paper can be achieved using the methods developed in this paper as starting point. 
One may think of driving by general expanding or by hyperbolic dynamical systems, 
and a treatment of larger classes of functions $\xi$.

%
%
%

\section{Approximation by step skew products}\label{s:approxbystep}

Although Theorems~\ref{t:expisinf} and \ref{t:birkhoff00} are formulated for $\hat{G} \in \hat{\mathcal{S}}$, the analysis mostly uses the formulation in terms of the skew product system $G$ on $\I \times \mathbb{R}$. 
We will approximate $G \in \mathcal{S}$ with step skew product systems driven by subshifts of finite type. This is done by using fine Markov partitions for $E_2$ of $\I$. 
Take a Markov partition $\mathcal{P}_N$ of $\I$ given by $K = 2^N$ partition elements 
\[
P_i = \left[\frac{i-1}{K}, \frac{i}{K}\right],
\] 
$1\leq i\leq K$. Note that 
\[E_2(P_i) = P_{2i -1} \cup P_{2i},\] indices taken modulo $K$.
Consider the resulting subshift of finite type $(\Sigma_{\mathcal{A}_N}, \sigma)$ with adjacency matrix $\mathcal{A}_N = (a_{ij})_{i, j=1}^K$ 
such that $a_{ij} = 1$ precisely if $j = 2i - 1$  or $j = 2i$ (modulo $K$). 
That is, $\Sigma_{\mathcal{A}_N} \subset \{1,\ldots,K\}^\mathbb{N}$ is given by
\[
 \Sigma_{\mathcal{A}_N} = \{  \omega \in \{1,\ldots,K\}^\mathbb{N} \; ; \; a_{\omega_i\omega_{i+1}} = 1 \textrm{ for all } i\}.
\]
Here $\omega = (\omega_i)_{i\in \mathbb{N}}$.


\begin{example}
For $N=1$,
\[
 \mathcal{A}_1 =  \left( \begin{array}{cc}  1 & 1 \\ 1 & 1  \end{array} \right),
\]
which is the full shift on two symbols $1,2$.
For $N=2$, $\mathcal{A}_2$ is the following $4\times 4$-matrix: 
\[
\mathcal{A}_2 =  \left( \begin{array}{cccc}  1 & 1 & 0 & 0 \\  0 & 0 & 1 &  1  \\  1 &  1 & 0 & 0 \\  0 & 0 & 1 &  1    \end{array} \right).
\]
It induces a subshift of finite type on sequences with four symbols $1,2,3,4$.
\end{example}

The shift $\sigma$ on $\Sigma_{\mathcal{A}_N}$ is primitive and in fact
\begin{align*}
\mathcal{A}_N^N &> 0
\end{align*}
($\mathcal{A}^N_N$ being the matrix with $1$ at every entry).

We use notation
\[
C^{0,\ldots,n-1}_{i_0,\ldots,i_{n-1}} = \{ \omega \in \Sigma_{\mathcal{A}_N} \; ; \; \omega_j = i_j, 0\le j < n \}.
\] 
Such sets are called cylinders; 
a cylinder with the first $n$ entries specified is called a cylinder of rank $n$.

The stochastic matrix $\Pi_N$ is given by 
\[
\Pi_N = \frac{1}{2} \mathcal{A}_N.
\]  
Denote the corresponding Markov measure on $\Sigma_{\mathcal{A}_N}$ by $\nu_N$,
where $\nu_N$ assigns equal measure $1/K$ to each cylinder of rank one $C^0_i$
(Appendix~\ref{a:subshift} contains more on this material).  

The following lemma contains the basic approximation result that we will use frequently in the following.
It starts by rewriting the skew product system on  $\I\times \mathbb{R}$  to a skew product system
on $\Sigma_{\mathcal{A}_N} \times \mathbb{R}$.  
The two systems are topologically semi-conjugate and they are measurably isomorphic, where we use Lebesgue measure on $\I\times \mathbb{R}$
and the product of Markov measure $\nu_N$ and Lebesgue measure on $\Sigma_{\mathcal{A}_N} \times \mathbb{R}$.

\begin{lemma}\label{l:approx}
For any given $N$,
a skew product system $G \in \mathcal{S}$ is measurably isomorphic to a skew product system 
$F: \Sigma_{\mathcal{A}_N} \times \mathbb{R} \to \Sigma_{\mathcal{A}_N} \times \mathbb{R}$,
\begin{align}\label{e:F}
F (\omega,x) &= (\sigma\omega, f_{\omega} (x)) = (\sigma \omega , x + \xi (\omega)+ R(\omega,x))
\end{align}
with
\[
\int_{\Sigma_{\mathcal{A}_N}}   \xi(\omega) \, d\nu_N (\omega) =0
\]
and 
\begin{align*}
|R(\omega,x)| &\le C e^{x},
\\
|R(\omega,x)| &\le C r_0,
\end{align*}
for some $C>0$.

The skew product system $F$
can be approximated by a skew product 
$F_N: \Sigma_{\mathcal{A}_N} \times \mathbb{R} \to \Sigma_{\mathcal{A}_N} \times \mathbb{R}$,
\begin{align*}
F_N (\omega,x) = (\sigma\omega, f_{N ,\omega} (x)) = (\sigma \omega , x + \xi_N (\omega_0)+ R_N(\omega_0,x)).
\end{align*}
Here 
\begin{align*}
\min_{\omega \in C^0_{\omega_0}}  \xi (\omega) \le \xi_N (\omega_0)  \le \max_{\omega \in C^0_{\omega_0}}  \xi (\omega). 
\end{align*}
with
\[
\int_{\Sigma_{\mathcal{A}_N}}   \xi_N(\omega) \, d\nu_N (\omega) = 0
\]
and the function $R_N$ can be any function so that
\begin{align*} 
 |R_N (\omega,x)| &\le C e^{x},
 \\ 
 |R_N (\omega,x)| &\le C r_0.
\end{align*}
There exists $C>0$ so that  
\begin{align*}
\left\vert   \xi (\omega) - \xi_N (\omega_0) \right\vert &\le C/2^N.
\end{align*}
\end{lemma}

\begin{remark}
The skew product system $F_N$ is a step skew product system:
the fiber map  $f_{N ,\omega}$ does not change for $\omega$ in cylinders of rank one,
and hence can be written as a function $f_{N,\omega_0}$. 
\end{remark}

\begin{proof}
For any $N$, we can map $\I$ to $\Sigma_{\mathcal{A}_N}$ by $I_N(x) = \omega$ with
\begin{align}\label{e:HSigmaT}
 \omega_i &= j \textrm{ if } E_2^i (x) \in \left[\frac{j-1}{2^N}, \frac{j}{2^N}\right). 
\end{align}
This yields the topological semi-conjugacy
\begin{align}\label{e:factor}
 \sigma \circ I_N &= I_N \circ E_2.
\end{align}
The map $H_N: \Sigma_{\mathcal{A}_N} \to \I$ given by
\[
 H_N (\omega) = \bigcap_{i\ge 0} E_2^{-i} (P_{\omega_i})
\]
is $\nu_N$-almost everywhere an inverse of $I_N$.
%
%
The system $\sigma: \Sigma_{\mathcal{A}_N} \to  \Sigma_{\mathcal{A}_N}$ with measure $\nu_N$ is therefore measurably isomorphic to
$E_2: \I \to \I$ with Lebesgue measure. 
With $G$ as in \eqref{e:GfromS},
let $F : \Sigma_{\mathcal{A}_N} \times \mathbb{R} \to \Sigma_{\mathcal{A}_N} \times \mathbb{R}$ be given by
\[
F (\omega,x) = (\sigma\omega, f_{\omega} (x)) = (\sigma \omega  , g_{H_N (\omega)} (x)).
\]
By \eqref{e:factor} we have
\begin{align*}
 F \circ (I_N,\mathrm{id}) &= (I_N,\mathrm{id}) \circ G.
\end{align*}
The skew product system $G: \I \times \mathbb{R} \to \I\times \mathbb{R}$ (using Lebesgue measure) is hence measurably isomorphic to
$F : \Sigma_{\mathcal{A}_N} \times \mathbb{R} \to \Sigma_{\mathcal{A}_N} \times \mathbb{R}$ (with measure the product of $\nu_N$ and Lebesgue measure).
Note that we use the term measurably isomorphic without demanding, as is often done, that the given measures are invariant. 

Recall from \eqref{e:gy} the notation $g_y (x) = x + \xi(y) + R(y,x)$.
We will accordingly write, by a slight abuse of notation, 
\[
f_{\omega} (x) = x + \xi (\omega)+ R(\omega,x).
\]

The skew product map $F$ is approximated by a skew product map
$F_N: \Sigma_{\mathcal{A}_N} \times \mathbb{R} \to \Sigma_{\mathcal{A}_N} \times \mathbb{R}$,
\begin{align*}
F_N (\omega,x) &= (\sigma \omega , f_{N,\omega} (x)) = (\sigma\omega , x + \xi_N (\omega) + R_N(\omega,x)),
\end{align*}
where $\xi_N (\omega)$ and $R_N (\omega,x)$ as function of $\omega$ are constant on the cylinder $C^0_{\omega_0}$.  
We choose $\xi_N (\omega) = \xi (H_N (\tilde{\omega}))$
for a fixed choice of $\tilde{\omega} \in C^0_{\omega_0}$.
We will also write $\xi_N (\omega_0)$ and $R(\omega_0,x)$ to emphasize the dependence on $\omega_0$ alone.
The function $R_N$ is a function satisfying bounds in Lemma~\ref{l:approx}.

In the above setup we get 
\begin{align*}
\int_{\Sigma_{\mathcal{A}_N}}   \xi(\omega) \, d\nu_N (\omega) &=0.
\end{align*}
By shifting the values $\tilde{\omega}$ in the cylinders $C^0_{\omega_0}$, we can achieve
\begin{align*}
\int_{\Sigma_{\mathcal{A}_N}}   \xi_N(\omega) \, d\nu_N (\omega) &=0.
\end{align*}

The bounds on $\xi-\xi_N$ follow from the facts that 
$\xi$ is $C^1$ and the diameter of each partition element $P_i$ is $1/2^N$. 
\end{proof}


Phrased in different words,
starting with a chaotic walk
\[
x_{n+1} = x_n + \xi (\sigma^n\omega) + R(\sigma^n\omega,x_n)
\]
on $\mathbb{R}$, with $\omega \in \Sigma_{\mathcal{A}_N}$, we approximate by a Markov random walk
\[
x_{n+1} = x_n + \xi_N (\omega_n) + R_N(\omega_n,x_n).
\]

Assume that $\xi$ is a smooth strictly increasing function on $\I$ with $\int_{\I} \xi(y)\,dy=0$. 
We need the following lemma, see 
Property \eqref{e:xM>L} in Appendix~\ref{a:stop}.

\begin{lemma}\label{l:left}
For $r_0$ small enough, and any $L>0$, there is $\omega$ from a set of positive probability so that for $x \in \mathbb{R}$, there is $n\in\mathbb{N}$, $f^n_\omega    (x) < x-L$ 
($f^n_\omega (x) > x+L$) and
$f^i_\omega (x) < x$ ($f^i_\omega(x) > x$) for all $0<i\le n$.
\end{lemma}

\begin{proof}
This is obvious from the monotonicity of $\xi$ and $\int_{\I} \xi(y)\,dy=0$.
\end{proof}

\section{Average return times}

Let us make explicit the convention used already in previous parts: 
we use $C$ to designate a generic constant depending only on data from the skew product system.
In particular, when considering approximations with Markov shifts driven by the shift on  $\Sigma_{\mathcal{A}_N}$ for different $N$, 
a constant $C$ is assumed to be uniform in $N$.

We start with a proposition on escape times from compact intervals.
Let $G \in \mathcal{S}$ as in \eqref{e:GfromS}.
Consider a compact interval $[A,B] \subset \mathbb{R}$ and let $x \in [A,B]$.
Define the escape time
\[
T_{[A,B]} (y) = \min \{ n >0 \; ; \; g^n_y(x) \not \in [A,B] \},
\]
taking $T_{[A,B]} (y)$ to be infinite if no such $n$ exists.


\begin{proposition}\label{p:aeescape}
%
 For Lebesgue almost all $y \in \I$,  $T_{[A,B]} (y) < \infty$.
 The average escape time from $[A,B]$ is finite:
 \[
 \int_{\I} T_{[A,B]} (y) \, dy < \infty.
 \] 
Moreover, both
\[
 p_A = Leb ( \{ y \; ; \;  x_{T_{[A,B]}}  (y) < A  \} )
\]
and $1 - p_A$ are positive.
\end{proposition}

\begin{proof}
We will make use of the topological semi-conjugacy of $E_2$ on $\I$ to the shift $\sigma$ on $\Sigma = \Sigma_{\mathcal{A}_1} = \{1,2\}^\mathbb{N}$.
This is by the map $H: \I \to \Sigma$ defined by $H (y) = \omega$ with $\omega = (\omega_i)_{i \in \mathbb{N}}$ and 
\[
 \omega_i = \left\{ \begin{array}{ll}  1, &  E_2^i(y) < 1/2, \\ 2, & E_2^i (y) \ge 1/2. \end{array}\right.
\]
We have
$\sigma \circ H= H \circ E_2$. In fact, $H$ provides a measurable isomorphism between $E_2: \I\to \I$ with Lebesgue measure and $\sigma: \Sigma \to \Sigma$
with Bernoulli measure $\nu$ (the product measure coming from equal probability $1/2$ for both symbols $1$ and $2$).

Consider the topologically semi-conjugate skew product system
\[
 F(\omega,x) = (\sigma \omega , f_\omega (x))
\]
on $\Sigma \times \mathbb{R}$. 
Our assumptions give the existence of $\zeta_0,\ldots,\zeta_{n-1}$ so that
$f^n_\omega (B) < A$ if $\omega\in C^{0,\ldots,n-1}_{\zeta_0,\ldots,\zeta_{n-1}}$. 
For $\nu$-almost every $\omega \in \Sigma$ one has a first entrance time 
\[
T_\zeta (\omega) = \min \{ i>0 \; ; \; \sigma^i \omega \in C^{0,\ldots,n-1}_{\zeta_0,\ldots,\zeta_{n-1}} \}.
\]
It is standard that the expected entrance time is finite, 
see for instance \cite[Section~17.3.2]{levperwil09} (see also \cite{ghahom17}):
\[
\int_{\Sigma}  T_\zeta (\omega) \, d\nu (\omega) < \infty.
\]
Indeed this is the expected time to throw for the first time the given finite sequence of symbols $\zeta_0,\ldots,\zeta_{n-1}$ when picking two symbols i.i.d.
with probability $1/2$ each.
It follows that 
\[
\int_{\I} T_{[A,B]} (y) \, dy < \infty.
\]

Analogous to the existence of $\zeta_0,\ldots,\zeta_{n-1}$, there are $\eta_0,\ldots,\eta_{m-1}$ with $f^m_\eta (A) > B$.
Also the expected entrance time to enter the cylinder $C^{0,\ldots,m-1}_{\eta_0,\ldots,\eta_{m-1}}$ is finite.
From this it is clear that $p_A$ and $1-p_A$ are positive.
\end{proof}

Take $G \in \mathcal{S}$ and consider the chaotic walk $g^n_y(x)$. 
Let $B \in \mathbb{R}$ and $x \in (-\infty,B]$. Define the escape time
 \begin{align}\label{e:TB}
 T_B (y) = \min \{ n >0 \; ; \; g^n_y(x) \not \in (-\infty,B] \}
 \end{align}
 with $T_B(y)$ infinite if no such $n$ exists.
The next two propositions together prove Theorem~\ref{t:expisinf}.
The first proposition, Proposition~\ref{p:aareturn} shows that almost surely, points escape from $(-\infty,B]$. 
The second proposition, Proposition~\ref{p:infinite}, establishes that the average escape time is infinite.  
In both statements, $r_0$ will be small.

\begin{proposition}\label{p:aareturn}
For Lebesgue almost all $y$, $T_B(y) < \infty$.
\end{proposition}

\begin{proof}
Apply Lemma~\ref{l:approx} to rewrite the chaotic walk as
\begin{align}\label{e:xn}
x_{n} &= f^{n}_{\omega}(x) = x_{n-1} + \xi(\sigma^{n-1}\omega) + R(\sigma^{n-1}\omega,x_{n-1}),
\end{align}
driven by the shift on $\Sigma_{\mathcal{A}_N}$ for some given $N$.
In this setting, \eqref{e:TB} becomes
\[
T_B(\omega) = \min \{n> 0\; ; \; f^n_{\omega}(x) \notin (-\infty, B]\}. 
\]
We must show
\begin{align*}
\nu_N \big( \{\omega \in \Sigma_{\mathcal{A}_N}\; ; \; T_B(\omega) < \infty\}\big) = 1.
\end{align*}

Divide $(-\infty, B]$ into three subintervals $(-\infty, A_2], [A_2, A_1]$ and $[A_1, B]$. 
If we start at $x \in [A_2,B]$ then by Proposition~\ref{p:aeescape}, with probability one some iterate will have left $[A_2, B]$: 
either through the right boundary point $B$ or through the left boundary point $A_2$.
In the first case the iterate has escaped from $(-\infty,B]$.
In the latter case, there may be a return to $[A_1, B]$ after which a further iterate may escape $(-\infty, B]$. 

We approximate the cocycle \eqref{e:xn} by Markov random walks, with different approximations on the (overlapping) intervals $(-\infty, A_1]$ and $[A_2, B]$.  
On $[A_2,B]$ we use an approximation by a Markov random walk driven by the shift on $\Sigma_{\mathcal{A}_{N_0}}$ for a sufficiently large but fixed $N_0$.
On $(-\infty,A_1]$ we use an approximation by a Markov random walk driven by the shift on $\Sigma_{\mathcal{A}_{N}}$, using increasingly fine partitions $\mathcal{P}_N$
for increasing $N$.
Also the boundary points $A_2$ and $A_1$ will depend on $N$:
we will take $A_1,A_2$ so that 
\[
 B-A_1,  A_1-A_2 = C_1 N
\]
for a suitable $C_1$, subject to conditions made explicit later on. 
Take a compact interval $[A_1,A_1 + E]$ lying inside $(A_2,B)$ with $E > \max_{\omega\in \Sigma_{\mathcal{A}_N}} |\xi(\omega)|$.
%
Without loss of generality, we may assume $x \in [A_1, A_1 + E]$ since $x > A_1$ will hold for $N$ large enough and any starting point in $[A_1 + E, B]$ will leave $(-\infty, B]$ in a finite number of iterates 
if $x$ does. 

As said, on $[A_2,B]$ we use an approximation by a Markov random walk driven by the shift on $\Sigma_{\mathcal{A}_{N_0}}$ for a sufficiently large but fixed $N_0$.
By Lemma~\ref{l:approx}, given $\omega \in \Sigma_{\mathcal{A}_{N_0}}$, the cocycle  \eqref{e:xn}
is approximated by a Markov random walk
\begin{align}\label{e:Mvn}
v_{n+1} &= v_n + \xi_{N_0}(\omega_n) + R_{N_0}(\omega_n, v_n).
\end{align}
Choose the approximation so that
\begin{align*}
\xi_{N_0} (\omega_0) + R_{N_0} (\omega_0,x) &\le \min_{\omega \in C^0_{\omega_0}}  \xi (\omega) + R (\omega,x). 
\end{align*}
Then
\[
x_0, x_1, \dots, x_n \in [A_2, B] \implies v_i \le x_i \text{ for all } 0\le i\le n+1.
\]
We have 
\[
 |\xi (\omega) - \xi_{N_0} (\omega_0)| \le C/2^{N_0} 
\]
for some $C>0$.
Further $|R_{N_0}| \le C r_0$ for some $C>0$.

On $(-\infty, A_1]$, applying 
Lemma~\ref{l:approx}, \eqref{e:xn}
is approximated by a Markov random walk
\begin{align}\label{e:wn}
v_{n+1} &= v_n + \xi_N(\omega_n) + R_N(\omega_n, v_n)
\end{align}
with
\begin{align*}
\xi_{N} (\omega_0) + R_{N} (\omega_0,x) &\le \min_{\omega \in C^0_{\omega_0}}  \xi (\omega) + R (\omega,x). 
\end{align*}
We have $ |\xi (\omega) - \xi_N (\omega_0)| \le C/2^N$ and  
$R_N(\omega, v) \le -C e^{-|A_1|}$ on $(-\infty, A_1]$.
Again 
\[
x_0, x_1, \dots, x_n \in (-\infty, A_1] \implies v_i \le x_i \text{ for all } 0\le i\le n+1.
\]


Let $v_0=x \in [A_1,A_1 + E]$. We can of course view the random walk driven by the shift on $\Sigma_{\mathcal{A}_{N_0}}$ 
as a random walk driven by the shift on $\Sigma_{\mathcal{A}_N}$.
Doing so we can say that for $\nu_N$-almost all $\omega \in \Sigma_{\mathcal{A}_N}$, there will be a finite first escape time $S_1$ 
of the random walk \eqref{e:Mvn} from $[A_2,B]$ (Proposition~\ref{p:aeescape}).
If $v_{S_1} < A_2$, then $v_{S_1}\in [A_2 - E, A_2]$ and
the random walk continues by following \eqref{e:wn}. There may be a return time $T_1 > S_1$ with $v_{T_1} > A_1$.
Then $v_{T_1}$ lies in $[A_1, A_1 + E]$.
After this we repeat by continuing with \eqref{e:Mvn}. 
Given $v_0 = x\in [A_1, A_1 + E]$ and $\omega\in \Sigma_{\mathcal{A}_N}$, we obtain a sequence of successive escape and return values:
\begin{gather*}
v_{S_1}\in [A_2 - E, A_2],\\
v_{T_1}\in [A_1, A_1 + E],\\
v_{S_2}\in [A_2 - E, A_2],\\
\vdots
\end{gather*}
stopping if either the walk starting from $v_{S_i}$ 
does not return to $[A_1,B]$, or if the walk starting from $v_{T_i}$ has escaped from $[A_2,B]$ through the right boundary point $B$. 
Formally $T_0=0$, 
\begin{align*}
S_i(\omega) &= \min \{n>T_{i-1}\; ; \; v_n(\omega) \notin [A_2, B] \},
\\
T_i(\omega) &= \min \{n>S_i\; ; \; v_n(\omega) \notin (-\infty, A_1] \}.
\end{align*}
Thus
 \[
 \nu_N \big( \{\omega \in \Sigma_{\mathcal{A}_N}\; ; \; T_B(\omega) < \infty\}\big) \geq \nu_N \big( \{\omega \in \Sigma_{\mathcal{A}_N}\; ; \; T_{B, v}(\omega) < \infty\}\big),
 \]
 with 
 \[
 T_{B, v}(\omega) := \min \{n> 0\; ; \; v_n(\omega) \notin (-\infty, B]\}.
 \]


Let 
\begin{align*}
p_{S_i} &= \nu_N \big( \{v_{S_i(\omega)} < A_2 \vert v_{\omega_{T_{i-1}}}, \omega_{T_{i-1}}\}  \big),
\\
p_{T_i} &= \nu_N \big( \{T_i(\omega) = \infty \vert v_{\omega_{S_i}}, \omega_{S_i} \}\big).
\end{align*}
We invoke the results in the appendices to obtain bounds on $p_{S_i}$ and $p_{T_i}$. 
We wish to apply Proposition~\ref{p:escapeprob} to get a bound on $p_{T_i}$.
Lemma~\ref{l:boundDV} and Lemma~\ref{l:boundV-} ensure bounds as required in Appendix~\ref{a:stop} for the Markov random walk.
The choices for $A_1,A_2$ ensure that \eqref{e:BG} holds for suitable $C_1$.
However, with $B$ as in \eqref{e:BG}, $A_1-A_2 \le B \le A_1-A_2 +E$ is not fixed as assumed in
Proposition~\ref{p:escapeprob}.
But the proof of Proposition~\ref{p:escapeprob}, see in particular \eqref{e:C5est}, gives 
%
$p_{T_i} \le C N / 2^N + C N e^{-A_1}$.
Choosing $C_1$ appropriately we find 
\[
 p_{T_i} \le CN / 2^N
\]
for some $C>0$. 

As $A_2-A_1 = C_1 N$, taking $C_1\ge 1$, $\sigma^{T_i-S_i}$ maps cylinders of rank one in $\Sigma_{\mathcal{A}_{N}}$ 
to a union of cylinders of rank one in $\Sigma_{\mathcal{A}_{N_0}}$.
From Proposition~\ref{p:a>0AB} (after reflecting through the origin) we infer $1 - p_{S_i} \ge  C \alpha_{N_0} e^{C \alpha_{N_0} N}$, or equivalently,
\[
p_{S_i} \le  1 - C \alpha_{N_0} e^{C \alpha_{N_0} N}
\] 
for some $C>0$. Here $\alpha_{N_0} <0$ and $ | \alpha_{N_0} | \le C r_0$ for some $C>0$, if $N_0$ is large enough.

Observe that $T_{B,v} (\omega)= \infty$ occurs if $S_1,T_1,\ldots,S_n$ are bounded for some $n\ge 1$, $v_{S_i(\omega)} < A_2$ for $1\le i\le n$ 
and $T_n (\omega) = \infty$. 
For each $i\ge 1$, $p_{S_i}\le p$ and $p_{T_i} \le q$ with
\begin{enumerate}
\item $p$ is the probability to leave $[A_2, B]$ through $A_2$ if we start at $A_1$,
\item $q$ is the probability to stay in $(-\infty, A_1]$ if we start at $A_2$.
\end{enumerate}
We have
\begin{align*}
 p &\le 1 - C \alpha_{N_0} e^{C \alpha_{N_0} N},
 \\
 q &\le C N/ 2^N.
\end{align*}
Therefore,
\begin{equation}\label{e:nullset}
\nu_N \big( \{\omega \in \Sigma_{\mathcal{A}_N}\; ; \; T_{B, v}(\omega) = \infty\}\big) \le pq ( 1 + \sum_{i=1}^\infty p^i (1-q)^i) < \dfrac{pq}{1-p}.
\end{equation}

%

The right hand side of \eqref{e:nullset} goes to zero as $N\to \infty$. 
Hence $T_{B, v}(w)<\infty$ for almost all $\omega$ and the same is true for $T_B(\omega)$. 
\end{proof}

The following proposition states that the average return time of a point to a compact interval is infinite.
The arguments in the proof are reminiscent of the proof of Proposition~\ref{p:aareturn}, but uses an approximation by step skew product systems
from above instead of from below.  A similar strategy was employed in \cite[Lemma~5.2]{ghahom17} for nonhomogeneous random walks with i.i.d. steps.

\begin{proposition}\label{p:infinite}
\[
\int_{\I} T_B(y) \, dy = \infty.
\] 
\end{proposition}

\begin{proof}
Given $x_0<x$, with positive probability one can find an $n$ so that  $f^n_y (x)<x_0$. See Lemma~\ref{l:left}. 
So we may have an assumption on a large enough distance  $B-x_0$ for a starting point $x_0$. 
There is for the same reason no loss in assuming that $B$ is a large negative number.

We make use of this by 
splitting an interval $(-\infty,B]$ in parts and considering the random walks separately in different parts.
 Take points $A_2 < A_1 < B$.  
 The mutual distances between these points will be determined later.
 We first consider escape of $x_k$ from $[A_2,B]$, for $x_0 \in [A_1,B]$, and in particular escape through the boundary $A_2$. 
 That will be followed by a second part of the argument treating a subsequent return to $[A_1,B]$. \\
 
\noindent {\sc Part $1$ (Initial point in $[A_1,B]$, escape from $[A_2, B]$).}
By Lemma~\ref{l:approx}, 
for a given Markov partition $\mathcal{P}_{N_0}$,  \eqref{e:F} can be written as the chaotic walk
\begin{align}\label{e:systemx1}
x_{k+1} &= x_k + \xi   (\sigma^k \omega) + R ( \sigma^k \omega,x_k)
\end{align}
on $\Sigma_{\mathcal{A}_{N_0}} \times \mathbb{R}$. 
For initial points in $[A_1,B]$ we compare \eqref{e:systemx1} with a random walk driven by the subshift for a given Markov  partition
$\mathcal{P}_{N_0}$ for some sufficiently large $N_0$ (conditions determining the size of $N_0$ will be made explicit below).
That is, we compare with
\begin{align*}
w_{k+1} &= w_k + \xi_{N_0}  (\omega_k) + R_{N_0} (\omega_k,x_k)  
\end{align*}
such that $x_k \le w_k$ for $x_0=w_0$, as long as $w_k \le B$. 
It follows from Lemma~\ref{l:approx}
that we may take $R_{N_0}$ so that
\[
R_{N_0} (\omega_k,x_k) \le \beta_{N_0} = C e^{B} + C/2^{N_0}. 
\]
We get $\beta_{N_0}$ small by 
choosing $B$ and $N_0$ large.


Define 
\[
T_{[A_2,B]} (y) = \min \{ n > 0 \; ; \; x_n \notin [A_2,B] \}
\]
and 
\[
S_{[A_2,B]} (\omega) = \min \{ n>0 \; ; \; w_n \notin [A_2,B] \}.
\]
We abbreviate this as $T(y)$ and $S(\omega)$. 
Write 
\begin{align*}
S_1 = \{ \omega \in \Sigma_{\mathcal{A}_{N_0}}  \; ; \;  w_{S (\omega)} < A_2  \}.
\end{align*}
Define $p_{A_2} = \nu(S_1)$
and $q_{A_2} = \mathrm{Leb} \{ y \in \I \; ; \;  x_{T(y)} < A_2 \}$.
Then $q_{A_2} \ge p_{A_2}$.

Proposition~\ref{p:a>0AB} yields a lower bound 
\[
p_{A_2} \ge C \alpha_1 e^{C \alpha_1 A_2}
\]
with $\alpha_1 = \beta_{N_0}$.
To apply the proposition we need a large enough distance $B  - x_0$ as expressed by \eqref{e:ABG}. As discussed earlier, we may assume this to hold.
In the lower bound, $\alpha_1$ is small if $r_0$ is small and $N_0$ large.
Hence 
 \[
 q_{A_2} \ge C \alpha_1 e^{C \alpha_1 A_2}.
 \] 

\noindent {\sc Part $2$ (Initial point in $(-\infty,A_2]$, escape from $(-\infty, A_1]$).}
For points that end up to the left of $A_2$ we estimate the time it takes to return to larger values, to the right of $A_1$.
  By \eqref{e:HSigmaT} for a given Markov partition $\mathcal{P}_{N}$,  \eqref{e:F} can be written as the random walk
  \begin{align}\label{e:systemN}
  x_{k+1} &= x_k + \xi  (\sigma^k \omega) + R( \sigma^k \omega,x_k)
  \end{align}
  on $\Sigma_{\mathcal{A}_{N}} \times \mathbb{R}$.
  We do this for increasing values of $N$. Assume given an initial point $x_0 < A_2$.
Consider the first escape time 
\[
U_{A_1} (\omega) = \min \{ n>0  \; ; \;  x_n(\omega) > A_1   \}.
\] 

A bound for $\mathbb{E} (U_{A_1})$ is derived as in Proposition~\ref{p:a>0}. Care must be taken since
the solution to the Poisson equation is unbounded in $N$. 
The random walk \eqref{e:systemN} is compared with
\begin{align*}
w_{k+1} &= w_k + \xi_N (\omega_k) + R_N (\omega_k,x_k) 
\end{align*}
on $\Sigma_{\mathcal{A}_N} \times \mathbb{R}$, 
such that $x_k \le w_k$ as long as $w_k \in (-\infty, A_1]$. 
We can bound 
\[
R_N (\omega_k,x_k) \le \beta_N = Ce^{A_1} + C/2^N.
\]
Consider the stopping time 
\[
V_{A_1} (\omega) = \min \{  n>0 \; ; \;  w_n (\omega) > A_1 \}.
\] 
Observe that $U_{A_1}(\omega) \ge V_{A_1} (\omega)$.
Proposition~\ref{p:a>0} yields 
$\mathbb{E} (V_{A_1}) \ge C / \beta_N$,
assuming $A_1 - A_2 \ge G$ where $G$ appears as condition \eqref{e:BG} that is needed to apply Proposition~\ref{p:a>0}. 
Lemma~\ref{l:boundDV} gives that we can take $G = CN$ for some $C>0$.  
We therefore take 
\begin{align*}
 A_2 &= -C_2 N,
 \\
 A_1 &= -C_1 N
\end{align*}
for uniform constants $C_1,C_2$ so that $A_1$, $A_2$ go to $-\infty$ as $N\to \infty$ and $A_1 - A_2 \ge G$. 
We conclude that under these conditions, for some $C>0$, 
\begin{align*}
\mathbb{E} (U_{A_1}) \ge C / \beta_N.
\end{align*}
The constant $C$ is uniform in $N$ and also in $\omega_{-1}$. \\


\noindent {\sc Part $3$ (Combining previous parts).}  
We combine the information in the previous parts on the two random walks in $(-\infty,A_1]$ and $[A_2,B]$ and let $N$ go to $\infty$ to prove the result.
For $x_0 < B$, let 
\[
T (y) = \min\{  n>0 \; ;  \; x_n > B \}.
\]
For the random walk $w_{n+1} = w_n + \xi_N (\omega_n) + R_N (\omega_n , x)$,
Part $1$ and Part $2$  above give
\begin{align*}
\mathbb{E}(T) &\ge q_{A_2} \mathbb{E} (U_{A_1})
\\
&\ge C e^{C \alpha_{N_0} A_2} (e^{C A_1} + 1 / 2^N)^{-1}.
\end{align*}
With our choices for $A_1,A_2$, this goes to infinity as  $N\to \infty$.
The proposition follows.
\end{proof}

\section{Intermittent time series: two neutral boundaries}

This section contains the proof of Theorem~\ref{t:birkhoff00two}.

\begin{proof}[Proof of Theorem~\ref{t:birkhoff00two}]
 On $\I \times \mathbb{R}$, define a class $\mathcal{C} = \mathcal{C}_d$ of differentiable functions
 \[
 \mathcal{C}_d =  \{ C \in C^1( \I , \mathbb{R}) \; ; \;  |C'| \le d \}.
 \]
For a skew product system of the form
$G_0(y,x) = ( E_2 (y) , x + \xi(y))$ we have
\[
 DG_0(y,x) = \left( \begin{array}{cc}  2 & 0 \\  \xi' (y) & 1 \end{array} \right).
\]
It follows that $DG_0$ maps a cone field of cones $\{ (v,u) \in \mathbb{R}^2 \; ; \; |u| \le d |v| \}$ with $d > \sup_{y\in \mathbb{T}} | \xi'(y) |$, inside itself. 
This property is easily seen to be robust under perturbations: for small $r_0$ the same holds true for perturbations $G$ of $G_0$ in $\mathcal{S}_{r_0}$.
%
Therefore, for suitable $d>0$,  $G$ maps the graph of a function in $\mathcal{C}_d$ into two curves
that are both the graph of a function in $\mathcal{C}_d$.

To continue we find it convenient to use, as in the proof of Proposition~\ref{p:aeescape}, the topological semi-conjugacy of $E_2$ on $\I$ to the shift $\sigma$
on $\Sigma = \Sigma_{\mathcal{A}_1} = \{1,2\}^\mathbb{N}$, $\sigma \circ I_1 = I_1 \circ E_2 $.
The space $\Sigma$ is equipped with Bernoulli measure $\nu$.
We henceforth consider functions in $\mathcal{D}_d$, which are the functions $D$ that can be written as 
$D (\omega) = C ( I_1 (y))$ for some $C \in \mathcal{C}_d$.

Given $\omega \in \Sigma$ and a graph $D_0 \in \mathcal{D}_d$, we obtain a sequence
of functions
\[
 D_n  (\eta) = f^n_{\omega_0\ldots\omega_{n-1}\eta} ( D_0 (\omega_0\ldots\omega_{n-1}\eta)),
\]
where $\omega_0\ldots\omega_{n-1}\eta$ stands for the concatenated sequence $\omega_0\ldots\omega_{n-1}\eta_0\eta_1\ldots$.
We have $D_n \in \mathcal{D}_d$ for all $n\ge 0$.
The interval $U\subset (0,1)$ corresponds to an interval $[-L,L] \subset \mathbb{R}$, with $L$ large. 
Start with $D_0$ so that $\textrm{graph}\, (D_0) \subset \I \times [-L,L]$.
Because $D_n \in \mathcal{D}_d$ for all $n\ge 0$, we can 
consider a sequence $D_{n_i}, D_{m_i}$ of curves that are first returns to $[-L,L]$ respectively its complement.
More precisely,
\begin{align*}
\textrm{graph}\, D_{n_i} &\subset \Sigma \times [-L,L], \\ 
\textrm{graph}\, D_{m_i} &\subset \Sigma \times (\mathbb{R}\setminus [-L,L])
\end{align*}
and 
\begin{align*}
n_{i+1} &= \min\{ n > m_{i+1} \; ; \; \textrm{graph}\, D_{n_{i+1}} \subset \Sigma \times [-L,L], \textrm{graph}\, D_{n_{i+1}-1} \not\subset \Sigma \times [-L,L]\},
\\
m_{i+1} &= \min \{ n > n_{i} \; ; \;  \textrm{graph}\, D_{m_{i+1}} \not \subset \Sigma \times [-L,L], \textrm{graph}\, D_{m_{i+1}-1} \subset \Sigma \times [-L,L]\}.
\end{align*}
Letting $n_0=0$, we thus obtain
a sequence 
\[
n_0 < m_1 < n_1 < m_2 < n_2 < \cdots
\] 
of successive escape times.
By results from the previous section, given $\omega \in \Sigma$ from a set of full Bernoulli measure, we obtain an infinite sequence of return times $n_i - m_i$.

Let $\bar{\nu}$ be normalized Bernoulli measure on $C^{0,\ldots,k-1}_{\omega_n,\ldots,\omega_{n+k-1}}$,
\[
 \bar{\nu} (A) = \nu(A) / \nu (C_{\omega_n,\ldots,\omega_{n+k-1}}^{0,\ldots, k-1})
\]
 for Borel sets $A \subset C_{\omega_n,\ldots,\omega_{n+k-1}}^{0,\ldots, k-1}$.
 On the graph of $D_n$ we can consider the measure $(\textrm{id},D_n)_* \nu$, which we also refer to as Bernoulli measure.
 Then
\begin{align}\label{e:nuinv}
 F^k_* \circ (\textrm{id},D_n)_*  \bar{\nu} &=  (id,D_{n+k})_*  \nu
\end{align}
In words, the push forward measure, under the skew product system $F^k$, of normalized Bernoulli measure on a cylinder of rank $k$ inside the graph of $D_n$,
equals Bernoulli measure on $D_{n+k}$.

Let 
\[
S(\omega) = \min \{ n>0 \; ; \; f^n_\nu (L) < L \;\textrm{for all}\; \nu \in C^{0,\ldots,n-1}_{\omega_0,\ldots,\omega_{n-1}} \}
\]
be the first return, for a cylinder, to $(-L,L)$ starting from $L$.
As $D_{m_i} (\omega) > L$ and the maps $f_\omega$ are monotone, the return time $n_i - m_i$
of $D_{m_i}$ to $[-L,L]$ is larger or equal than $S(\omega)$.
The return time $S$ defines a partition $\mathcal{Q} = \{ Q_i\}$ on $\Sigma$ so that $S$ is constant precisely on the partition elements $Q_i$.
Each $Q_i$ is a cylinder and the union of all cylinders has full Bernoulli measure $\nu$.

As remarked above, 
if 
\[T_i = \min \{ n>0 \; ; \; f^n_{\omega} (D_{m_i} (\omega)) < L \;\textrm{for all}\; \nu \in C^{0,\ldots,n-1}_{\omega_0,\ldots,\omega_{n-1}}  \},
\] 
then $T_i (\omega) \ge S(\omega)$.
The corresponding partition $\mathcal{R}_i$ of $\Sigma$, corresponding to cylinders on which $T_i$ is constant, is therefore the same or finer than $\mathcal{Q}$;
each partition element in $\mathcal{Q}$ is a union of one or more partition elements of $\mathcal{R}_i$.

Given a large integer $M$,
take the partition $\mathcal{O}$ of $\Sigma$ consisting of the cylinders $Q_i$ from $\mathcal{Q}$ up to rank $M$ and the remaining subset $\hat{Q}$. 
Note that 
\[
\int_{\cup_i Q_i} S (\omega) \, d\nu(\omega)   =   \sum_{i, Q_i \in \mathcal{O}} \nu(Q_i) S_i,
\]
where $S_i$ is the value of $S$ on $Q_i \in \mathcal{O}$.
We claim that $\int_{\Sigma}  S(\omega) \, d\nu (\omega) = \infty$.
To see this note that if $x_n (\omega) < L-d$, then $x_n (\eta)<L$ for any $\eta \in C^{0,\ldots,n-1}_{\omega_0,\ldots,\omega_{n-1}}$.
The claim follows from this and Proposition~\ref{p:infinite}.
Since $\int_{\Sigma}  S(\omega) \, d\nu (\omega) = \infty$, we find that for each $C>0$ there is $M$ so that
\[\sum_{i, Q_i \in \mathcal{O}} \nu(Q_i) S_i > C.\]

Consider the following stochastic process. Pick $\omega_0,\omega_1,\ldots$ at random, independently from two symbols $1,2$ with probabilities $1/2$ each. 
Given $\omega_0,\ldots,\omega_{n-1}$, we have return times $n_0 < m_1 < \cdots$ up to the largest number $n_k$ or $m_k$ that is at most $n$.
To fix thoughts, assume given $\omega_0,\ldots,\omega_{m_i-1}$ defining the function $D_{m_i}$.
Continue with the next random variables $\omega_{m_i}, \omega_{m_i+1},\ldots$.
As these are picked independently from the past symbols, the probability for $\omega_{m_i}\omega_{m_i+1}\cdots$ to end up in $Q_i$ equals $\nu(Q_i)$.
On $Q_i$ we know that $n_i - m_i \ge S_i$.
By \eqref{e:nuinv} we can use this stochastic process description and apply Kolmogorov's strong law to get for $\nu$-almost every $\omega$,
\[
\lim_{n\to \infty} \frac{1}{n} \sum_{j=1}^{n-1} \mathbbm{1}_{Q_i} (\omega_{n_j}) = \nu(Q_i).
\]

We must also consider the escape from $[-L,L]$, which proceeds similarly.
Recall that $x_n(\omega) > L+d$ implies 
that $x_n(\eta)>L$ for all $\eta \in C^{0,\ldots,n-1}_{\omega_0,\ldots,\omega_{n-1}}$.
Consider 
\[
v_{n+1} = \min \{ f_{\sigma^n \omega} (v_n) , L + d \}.
\] 
starting at $v_0 = L$. Let
\[
S_{-L} (\omega) = \min \{ n>0 \; ; \;  v_n(\nu) < -L \;\textrm{for all}\; \nu\in C^{0,\ldots,n-1}_{\omega_0,\ldots,\omega_{n-1}}   \}.
\] 
As in Proposition~\ref{p:aeescape},
\[
 \int_{\I} S_{-L} (\omega)  \, d\nu (\omega) < \infty.
\]
Corresponding statements hold for escape through $L$:
consider 
\[
w_{n+1} = \max \{ f_{\sigma^n \omega} (w_n) , - L - d \}
\] 
with $w_0 = L$. For
\[
S_{L} (\omega) = \min \{ n>0 \; ; \;  v_n(\nu) > L \;\textrm{for all}\; \nu\in C^{0,\ldots,n-1}_{\omega_0,\ldots,\omega_{n-1}}   \},
\] 
we have
\[
 \int_{\I} S_{L} (\omega)  \, d\nu(\omega) < \infty.
\]
Define 
\[
S = \min \{ S_{-L}, S_L \}
\]
and note
%
\begin{align}\label{e:S-LLfinite}
 \int_{\I} S (\omega) \, d\nu(\omega) < \infty.
\end{align}
If 
\[T_i = \min \{ n>0 \; ; \; f^n_{\omega} (D_{n_i} (\omega)) \not\in [-L,L] \},
\] 
then $T_i (\omega) \ge S(\omega)$.

We follow previous reasoning.
The function 
$S$
defines a partition $\mathcal{Q} = \{ Q_i\}$ on $\Sigma$ so that $S$ 
is constant on each partition element $Q_i$.
Each $Q_i$ is a cylinder and the union of all cylinders has full Bernoulli measure $\nu$.
On each curve $D_{m_i}$ the escape time from $(-L,L)$ is larger than or equal to $S$.
Given $\omega \in \Sigma$ from a set of full Bernoulli measure, we obtain an infinite sequence of escape times $n_i - m_i$.
Given a large integer $M$,
take the partition $\mathcal{R}$ of $\Sigma$ consisting of cylinders $R_i$ from $\mathcal{Q}$ up to rank $M$ and the remaining subset $\hat{R}$. 

Let $S_i$ is the value of $S$ on $R_i$.
By \eqref{e:S-LLfinite},
for any $\varepsilon$ one can find $M$ large enough so that  
\[
\int_{\Sigma} S (\omega) \, d\nu(\omega)   \le \sum_i \nu(R_i) S_i + \varepsilon.
\]
Again by Kolmogorov's strong law, for $\nu$-almost every $\omega$,
\[
\lim_{n\to \infty} \frac{1}{n} \sum_{j=1}^{n-1} \mathbbm{1}_{R_i} (\omega_{n_j}) = \nu(R_i).
\]

We conclude that for almost all $\omega \in \Sigma$,
\[  
 \lim_{n\to \infty}  \frac{1}{n} \sum_{i=0}^{n-1}  (n_{i} - m_i) = \infty  
\]
and 
\[  
 \lim_{n\to \infty}  \frac{1}{n} \sum_{i=0}^{n-1}  (m_{i+1} - n_i)  < C  
\]
for some $C>0$.
The theorem follows.
\end{proof}

\section{Intermittent time series: a neutral and a repelling boundary}

For $B \in \mathbb{R}$ and fixed $x_0>B$, let
\[
T_B = \min \{ n>0 \; ; \; x_n < B \}.
\]

\begin{proposition}\label{p:Ainfty}
 If $a>0$ then 
 \[
  \int_{\I} T_B (y) \, dy < \infty.
 \]
\end{proposition}

\begin{proof}
We use an approximation of $x_{n+1} = f_\omega (x_n)$,
\begin{align}\label{e:v-p} 
v_{n+1} &= f_{N_0,\omega_n} (v_n)
 \end{align}
directed by $\omega \in \Sigma_{\mathcal{A}_{N_0}}$,
so that $x_n \le v_n$ when $x_0=v_0$, as long as $v_n > B$.
The condition $a > 0$ allows us the following formulas:
on $[B,\infty)$ we can write
$f_{\omega} (x) = x + \eta (\omega) + S(\omega , x)$ with 
\[\int_{\Sigma_{\mathcal{A}_N}} \eta (\omega) \, d\nu (\omega) = 0\] and \[\lim_{x\to \infty} R (\omega,x) = -a < 0.\]
Write \eqref{e:v-p} as 
\[
 v_{n+1} = v_n + \eta_{N_0} (\omega_n)  + S_{N_0} ( \omega_n, v_n)
\]
driven by the shift on $\Sigma_{\mathcal{A}_{N_0}}$.
We can take 
\[\int_{\Sigma_{\mathcal{A}_{N_0}}}   \eta_{N_0} (\omega) \, d\nu (\omega) = 0\] 
and 
\[\lim_{x\to \infty} S_{N_0} (\omega,x) = -a < 0.\]
For $\tilde{L}$ large we thus find $|S_{N_0} (\omega,x) +a| < a/2$ on $[\tilde{L},\infty]$.

With $L > \tilde{L}$, assume $z_0$ is in $[L,\infty)$. 
Let $S_L$ be the escape time of $z_n$ out of $[L,\infty)$;
\[
S_{\tilde{L}} (\omega) = \min \{ n >0 \; ; \; z_n < \tilde{L} \} . 
\]
By Proposition~\ref{p:a>0}, assuming that $L-\tilde{L}$ is large enough, 
we get 
\[
\mathbb{E} (S_{\tilde{L}}) \le C.
\] 
An escape out of $[B,\infty)$ is realized by a finite number of passages through $[L,\infty)$, starting and ending with iterates in $[B,\tilde{L}]$, 
followed by the escape to $(-\infty,B)$.
The expected escape times for these escapes from $[B,L]$, through either the left or right boundary, and from $[L,\infty)$ are bounded. 
Therefore, with $p_B$ bounding the probability to leave through $L$ for an escape from $[B,L]$ and starting point in $[B,\tilde{L}]$,
\[
\mathbb{E} (T_{B})  \le C \sum_{i=1}^\infty i p_{B}  (1-p_{B})^i < \infty.
\]
\end{proof}
%
%
%

\begin{proof}[Proof of Theorem~\ref{t:birkhoff00}]
Theorem~\ref{t:birkhoff00} is proved in the same way as Theorem~\ref{t:birkhoff00two}, invoking Proposition~\ref{p:Ainfty}. We leave the details to the reader.
\end{proof}

\appendix




\section{Chaotic walks}\label{a:chaoticwalk}

For $\omega = (\omega_i)_{i \in \mathbb{N}}$ in $\Sigma = \{1,2\}^\mathbb{N}$,  denote by $\sigma$ the left shift operator
$(\sigma \omega)_i = \omega_{i+1}$ acting on $\Sigma$.
Consider the standard random walk $x_n (\omega)$ given by
\[
x_{n+1} = x_n + \xi (\sigma^n \omega)
\] 
on $\mathbb{R}$, where for
$\omega \in \Sigma$ the step $\xi (\omega)$ is defined by
\[
\xi (\omega) = \left\{ \begin{array}{ll} 1, & \omega_0 = 1, \\ -1, & \omega_0 = 2.   \end{array}\right.
\]  
Introduce the skew product system $F: \Sigma \times \mathbb{R} \to   \Sigma \times \mathbb{R}$,
\[
 F(\omega,x) = (\sigma \omega , x + \xi (\omega)).
\]
Then the fiber coordinates in $\mathbb{R}$ of $F^n (\omega,x)$ equal $x_n (\omega)$ with initial condition $x_0 (\omega) = x$. 
The natural measure on $\Sigma$ is the Bernoulli measure $\nu$ given  probabilities $1/2, 1/2$ for the symbols $1,2$.

Since the shift $\sigma: \Sigma \to \Sigma$ is topologically semi-conjugate to the doubling map
$E_2 : \I \to \I$, the skew product system $F$ is topologically semi-conjugate
to $G: \I \times \mathbb{R} \to \I \times \mathbb{R}$,
\[
 G(y,x) = (E_2 (y) , x + \xi (y)),
\]
where now 
\[
\xi (y) = \left\{ \begin{array}{ll} 1, & y \in [0,1/2), \\ -1, & y \in [1/2,1].   \end{array}\right.
\]  
In fact, $\sigma: \Sigma \to \Sigma$ with $\nu$ as invariant measure is measurably isomorphic to
$E_2 : \I \to \I$ with Lebesgue measure.
Note that
\[
 G^n(y,x) = (E_2^n (y) , x + \sum_{i=0}^{n-1} \xi (E_2^i(y))).
\]

Going the other direction, also for other maps $\xi: \I \to \mathbb{R}$, one can identify $G$ on $\I \times \mathbb{R}$ 
with a skew product system $F$ on  $\Sigma \times \mathbb{R}$. 
Typically the steps $\xi(\omega)$ will depend on the entire sequence $\omega$.  
Following \cite{heaplaham94} we refer to a cocycle
\[
x \mapsto x + \sum_{i=0}^{n-1} \xi (E_2^i(y)),
\]
i.e. the fiber coordinate of $G^n (y,x)$,  as a chaotic walk.
We refer to \cite{sch77} for lecture notes on cocycles driven by ergodic transformations.

\section{Markov random walks}\label{a:subshift}

This appendix and the next develop material for random walks driven by subshifts of finite type. 
Appendix~\ref{a:fordoubling} will specialize to subshifts coming from Markov partitions for
doubling maps.
%


Write $\Omega$ for the finite set of symbols $\{ 1,\ldots,K\}$.
Let $\mathcal{A}=(a_{ij})_{i,j=1}^K$ be a matrix with $a_{ij}\in \{0,1\}$.
Associated to $\mathcal{A}$ is the set
$\Sigma_\mathcal{A}$  of bilateral sequences $\omega=(\omega_n)^{\infty}_{0}$ composed of symbols in $\Omega$ and
with adjacency matrix $\mathcal{A}$:
\[a_{\omega_n\omega_{n+1}}=1\] for all $n \in \mathbb{N}$.
Let 
$(\Sigma_\mathcal{A},\sigma)$ be the subshift of finite type 
on  $\Sigma_\mathcal{A}$.
The map $\sigma$ shifts every sequence $\omega \in \Sigma_\mathcal{A}$ one step to 
the left, $(\sigma \omega)_i = \omega_{i+1}$. 
The space $\Sigma_\mathcal{A}$ will be endowed with the product topology.
We assume that $\mathcal{A}$ is 
primitive,
i.e. 
\[
\exists n_0\in \mathbb{N} ~ \forall i,j \in \Omega ~ (\mathcal{A}^{n_0})_{ij}>0.
\]
This implies that
the subshift $\sigma$ is topologically mixing.

Let $\Pi=(\pi_{ij})_{i,j=1}^{K}$ be a right stochastic matrix,
i.e.
$\pi_{ij} \ge 0$ and $\sum_{j=1}^K \pi_{ij}=1$, 
such that 
$\pi_{ij}=0$ precisely if $a_{ij}=0$.
By the
Perron-Frobenius theorem for stochastic matrices, 
there 
exists 
a unique positive left eigenvector $p=(p_1,...,p_K)$ for $\Pi$ that corresponds to the eigenvalue $1$; 
i.e.
\begin{align*}
\sum\limits_{i=1}^K p_i \pi_{ij} &= p_j,
\end{align*}
for all $j \in \Omega$.
We assume that $p$ is normalized so that it is a probability vector, $\sum_{i=1}^K p_i=1$.
The distribution $p$ is the stationary distribution on $\Omega$.

For a finite word $\omega_{k_1}...\omega_{k_n}$, $k_i \in \mathbb{Z}$, 
the cylinder $C^{k_1,\ldots,k_n}_{\omega_{k_1},\ldots,\omega_{k_n}}$ (we will also use the notation $C^{k_1,\ldots,k_n}_{\omega}$) is the set
\[
C^{k_1,\ldots,k_n}_{\omega_{k_1},\ldots,\omega_{k_n}}= \{\omega' \in \Sigma_\mathcal{A}\; ; \; \omega'_{k_i}= \omega_{k_i},\; \forall 1\le i\le n\}. 
\]
As cylinders form a countable base of the topology on $\Sigma_\mathcal{A}$, Borel measures on
$\Sigma_\mathcal{A}$ are determined by their values on the cylinders.
A Borel measure $\nu$ on $\Sigma_\mathcal{A}$ is called a 
Markov measure
constructed from the distribution $p$ and the transition probabilities $\pi_{ij}$,
if
for every $\omega \in \Sigma_\mathcal{A}$ and $k \le l$, 
\[
\nu(C^{k,\ldots,l}_\omega) = p_{\omega_k} \prod_{i=k}^{l-1} \pi_{\omega_i \omega_{i+1}}.
\]
The measure $\nu$ is invariant under the shift map $\sigma$,
it is ergodic
and 
$\operatorname{supp}(\nu)=\Sigma_\mathcal{A}$.
From now on,
we consider a fixed ergodic Markov measure $\nu$ on $\Sigma_\mathcal{A}$.

\subsection{Poisson equation}

The Poisson equation is a means to calculate stopping times for Markov random walks.  
See \cite[Chapter 17]{meytwe93} and also \citep{fuhzha00}. 

Consider a random walk on $\mathbb{R}$,
\[ w_{n+1} = w_n + \xi (\sigma^n \omega), \]
driven by $\omega = (\omega_n)_{n\ge 0}$ from $\Sigma_{\mathcal{A}}$.
In the setting here, $\xi (\sigma^n \omega)$ is a function of $\omega_n$ alone.
Write $\mathcal{F}_k$ for the $\sigma$-algebra on $\Sigma_\mathcal{A}$ generated by cylinders $C_{\omega_0,\ldots,\omega_{k-1}}^{0,\ldots,k-1}$
of rank $k$.  Observe that a $\mathcal{F}_1$-measurable function $\phi:\Sigma_\mathcal{A} \to \mathbb{R}$ 
is constant on cylinders of rank one. 

Introduce expectation operators, for integrable functions $h : \Sigma_\mathcal{A} \to \mathbb{R}$,
\[
P h = \mathbb{E} ( h \circ \sigma \vert \mathcal{F}_1)
\]
and
\[
P_\nu h = \mathbb{E} (h).
\]
%
%
Thus
\[
P h (\omega) =  \frac{1}{p_{\omega_0}} \int_{C^0_{\omega_0}}   h (\sigma \omega) \, d \nu( \omega) 
\]
(recall $p_{\omega_0} = \nu(C^0_{\omega_0})$)
and
\[
P_\nu h  = \int_{\Sigma_\mathcal{A}}  h (\omega ) \, d\nu (\omega).
\]

\begin{lemma}\label{l:poisson}
There is a unique $\mathcal{F}_1$-measurable function $\Delta : \Sigma_\mathcal{A} \to \mathbb{R}$ that solves
\begin{align}\label{e:Poisson}
P \Delta ( \omega) - \Delta ( \omega)  &=  P \xi (\omega)   - P_\nu \xi
\end{align}
and satisfies
\begin{align*}
P_\nu \Delta &= 0.
\end{align*}
\end{lemma}

\begin{proof}
This is a special case of the theory developed in \cite[Chapter 17]{meytwe93}. 
We identify a $\mathcal{F}_1$-measurable function $h$ with the vector in $\mathbb{R}^K$ of its values on cylinders $C_{i}^0$, $i=1, \dots, K$.
We denote this vector also by $h= (h_i)_{i=1}^K$. 
Then $P h (\omega)$, for $\omega \in C^0_i$, equals $\sum_{j=1}^K \pi_{ij} h_j$.
Now
\eqref{e:Poisson} becomes
\begin{align}\label{e:poisson}
\Pi  \Delta - \Delta   &=  \Pi \xi  - P_\nu \xi,
\end{align}
to be solved for the  vector  $\Delta = (\Delta_i)_{i=1}^K$  (here $P_\nu \xi = \sum_{i=1}^K p_i \xi_i$).

By the Perron-Frobenius theorem, $\Pi - id$ has a one dimensional kernel and a codimension one invariant space 
$\{ h \vert P_\nu h = 0\}$ on  which $\Pi - id$ acts invertible.
Hence \eqref{e:poisson} can be solved since $P_\nu ( P \xi - P_\nu \xi) = 0$.
The demand $P_\nu \Delta = 0$ brings uniqueness of the solution. 
\end{proof}

Define
\begin{align*}
u_n &= w_n - n P_\nu \xi - \Delta(\sigma^{n-1}\omega) + \Delta(\omega).
\end{align*}
We get that 
\[
u_{n+1} = u_n + \xi(\sigma^n \omega) - P_\nu \xi - \Delta (\sigma^n \omega) + \Delta (\sigma^{n-1} \omega)
\]
is cohomologous to
$y_{n+1} = y_n + \xi (\sigma^n \omega) - P_\nu \xi$.

\begin{lemma}\label{l:martingale}
$u_n$ is a martingale with respect to $\mathcal{F}_n$.
\end{lemma}


\begin{proof}
Calculate
\begin{align*}
\mathbb{E} (u_{n+1} \vert \mathcal{F}_n) 
&= 
\mathbb{E} ( u_n + \xi (\sigma^n \omega) - P_\nu \xi - \Delta (\sigma^n \omega) + \Delta (\sigma^{n-1} \omega) \vert \mathcal{F}_n) 
\\
&=
u_n + P \xi (\sigma^{n-1}\omega) - P_\nu \xi - P \Delta (\sigma^{n-1} \omega) + \Delta (\sigma^{n-1} \omega)
\\
&= u_n,
\end{align*}
by Lemma~\ref{l:poisson}.
\end{proof}


\begin{example}
Consider the topological Markov chain given by
\[
A = \left( \begin{array}{cc} 1 & 1 \\ 1 & 0 \end{array} \right)
\]
and the probability matrix
\[
\Pi = \left( \begin{array}{cc} 1/2 & 1/2 \\ 1 & 0 \end{array} \right).
\]
For this, \[p = \left( \begin{array}{c} 2/3 \\ 1/3 \end{array}\right)\] satisfies $p^T = p^T \Pi$.
Consider the displacement vector 
\[\xi = \left( \begin{array}{c} -1 \\ 2 \end{array}\right).\] 
A solution  $\Delta$ of $\Pi \Delta - \Delta = \Pi \xi$ is given by
\[\Delta = \left( \begin{array}{c} -1/3 \\ 2/3 \end{array}\right).\] 

With $x_{n+1} = x_n + \xi_n$ and $u_n = x_n - \Delta(\sigma^{n-1} \omega) + \Delta (\omega)$,
and hence, $u_{n+1} = u_n + \xi_n - \Delta (\sigma^n \omega) + \Delta (\sigma^{n-1} \omega)$,
we find
\begin{itemize}
\item
if $\omega_{n-1} = 1, \omega_n = 1$:  $x_{n+1} = x_n -1$ and $u_{n+1} = u_n - 1$,
\item
if $\omega_{n-1} = 1, \omega_n = 2$:  $x_{n+1} = x_n +2$ and $u_{n+1} = u_n + 1$,
\item
if $\omega_{n-1} = 2, \omega_n = 1$:  $x_{n+1} = x_n -1$ and $u_{n+1} = u_n$,
\end{itemize}
the possibility $\omega_{n-1} = 2, \omega_n = 2$ being forbidden.
\end{example}

\section{Stopping times}\label{a:stop}

Using the material from Appendix~\ref{a:subshift} we state and prove various bounds on stopping times for Markov random walks.
We focus on bounds that are needed in the main text.
References \citep{krasza84,fuhzha00} also deal with stopping times for Markov random walks, but 
from a different angle.
As in \citep{fuhzha00} we construct martingales by
making use of the Poisson equation.

Deviating slightly from the notation in the previous appendix,
we consider a random walk
\begin{align*}
w_{n+1} &= w_n + \xi (\sigma^n \omega) + \alpha
\end{align*}
with $\omega \in \Sigma_\mathcal{A}$ and 
\[
P_\nu \xi = 0.
\]
In this setup, $\alpha$ is the average drift. The reader can think of $\alpha$ being small.
We write
\begin{align}\label{e:vnwn}
v_n &= w_n - \Delta(\sigma^{n-1} \omega) + \Delta (\omega)
\end{align}
and find
\begin{align*}
v_{n+1} &=  v_n + \xi (\sigma^n \omega) - \Delta (\sigma^n \omega) + \Delta (\sigma^{n-1} \omega) + \alpha.
\end{align*}
Note that
\begin{align*}
 v_n &= v_0 + \sum_{i=0}^{n-1} \xi(\sigma^i \omega) - \Delta(\sigma^{n-1}\omega) + \Delta (\omega) + n\alpha.
\end{align*}

We introduce notation 
$\xi (\omega_n) = \xi (\sigma^n \omega)$, $\Delta (\omega_n) =  \Delta (\sigma^n \omega)$
and 
$\zeta(\omega_{n-1},\omega_n) = \xi (\omega_n) - \Delta (\omega_n) + \Delta (\omega_{n-1})$.
%
%
Thus
\begin{align}\label{e:vn} 
v_{n+1} &= v_n + \zeta (\omega_{n-1},\omega_n) + \alpha.
\end{align}
Let also $u_n = v_n - n \alpha$ so that $u_{n+1} = u_n + \zeta (\omega_{n-1},\omega_n)$.
By Lemma~\ref{l:poisson}, for each $\omega_{k-1}$, 
\begin{align} 
\label{e:pzeta=0}
\sum_i \pi_{\omega_{k-1} i} \zeta (\omega_{k-1} ,i )  &=  0.
\end{align}

Let $G$ be such that 
\begin{align}\label{e:deltabound}
\vert \Delta (i) - \Delta(j) \vert &\le G
\end{align}
%
%
for all $i,j$.
We also assume explicit bounds: there are positive constants $V^-,V^+,D$ so that 
for all $i,j$ with $\pi_{ij} >0$, 
\begin{align}\label{e:stepbound}
\vert \zeta (i,j) \vert &\le D
\end{align}
and for all $i$,
\begin{align}\label{e:varbound}  
V^- <  \sum_{j=1}^K  \pi_{ij}  \zeta^2 (i,j) < V^+.
\end{align}

Throughout we assume a recurrence property: 
for any $L>0$ ($L<0$) there exists $\omega\in \Sigma_\mathcal{A}$ and $M \in \mathbb{N}$ so that, with $w_0 = 0$, 
\begin{align}\label{e:xM>L}
w_M &> L (w_M <L).
\end{align}
Together with $\mathcal{A}^{n_0} >0$ for some $n_0 \ge 0$ we find the existence of 
periodic symbol sequences $(\omega_0,\ldots,\omega_{k-1})^\infty \in \Sigma_\mathcal{A}$ for which $w_k > w_0$ (or $w_k<w_0$). 
A condition of this type is needed to avoid examples of the following type.

\begin{example}
Let
\[
\Pi = \left( \begin{array}{ccc}  1/2 & 1/2 & 0 \\ 0 & 0 &  1 \\ 1/2 & 1/2 & 0  \end{array} \right), \; \xi = \left( \begin{array}{c}  0 \\ 1 \\ -1  \end{array} \right)
\] 
Note that the associated adjacency matrix $\mathcal{A}$ is primitive with $\mathcal{A}^3 >0$, 
but $w_n \in \{0,1\}$ for all $n>0$, $\omega\in \Sigma_\mathcal{A}$ if $w_0=0$. 
\end{example}

As a consequence of condition \eqref{e:xM>L}, the probability to escape from a compact interval $[A,B]$ is always one.
We state this in the following lemma.
Consider a compact interval $[A,B]$  and define for $w_0 \in [A,B]$,
\[
T_{[A,B]} = \min\{ n >0 \; ; \; w_n \not\in [A,B] \}.
\]

\begin{lemma}\label{l:Tfinite}
For $\nu$ almost all $\omega \in \Sigma_\mathcal{A}$, $T_{[A,B]} < \infty$.
\end{lemma}

\begin{proof}
This is immediate from the observation that property \eqref{e:xM>L} implies the existence of a finite symbol sequence for which any point in $[A,B]$ is mapped outside the interval. The corresponding cylinder has positive $\nu$ measure. By ergodicity of $\sigma$, for $\nu$ almost all $\omega \in \Sigma_\mathcal{A}$
its positive orbit has points in this cylinder.  
\end{proof}

We will discuss escape times from bounded and unbounded intervals for different signs of the average drift $\alpha$. 
The principal technique is Doob's optional stopping theorem for which we refer to e.g. \cite[Theorems~VII.2.1~and~VII.2.2]{shi84}.
We state the theorem under conditions relevant to our setting.

\begin{theorem}[Doob's optional stopping theorem]
 Let $u_n$ be a martingale (or submartingale); $\mathbb{E} (u_{n+1} \vert \mathcal{F}_n) = u_n$ (or $\ge u_n$).
 Let $T$ be a stopping time.
 Suppose that either
 \begin{enumerate}
  \item $|u_{n+1} - u_{n}|$ is uniformly bounded and $\mathbb{E} (T) < \infty$, or,
  \item $|u_n|$ is uniformly bounded.
 \end{enumerate}
%
Then
\[
 \mathbb{E}(u_T ) = \mathbb{E}(u_0) \;
 (\textrm{or} \ge \mathbb{E}(u_0)).
\] 
\end{theorem}


The following subsections \ref{ss:>0}, \ref{ss:=0}, \ref{ss:<0} treat expected escape times and 
probabilities of escape for positive, zero, and negative drift $\alpha$ respectively.
One may consider different initial distributions for $\omega_0$. 
We will take $\omega_0$ from the stationary distribution, or 
assume $\omega_{-1}$ is given and 
$\omega_0$ is from the distribution that gives probability $\pi_{\omega_{-1} \omega_0}$ of picking $\omega_0$.
The formulations of the results then apply to both settings without change.
%
That is, the derived bounds for expected escape times and 
probabilities of escape involve constants depending on $G$ from \eqref{e:deltabound} $D$ from \eqref{e:stepbound} and $V^-,V^+$ from \eqref{e:varbound},
but not on the initial distribution for $\omega_0$.

\subsection{Positive drift \texorpdfstring{$\alpha>0$}{alpha>0}}\label{ss:>0}

We start with escape from bounded intervals $[A,B]$. For definiteness we assume $A < 0$, $B>0$ and $w_0 = 0$. 
By conjugating with a translation one may reduce to this situation.
 
Recall that
\[
T_{[A,B]} = \min \{ n>0 \ \;\vert \; w_n \not \in  [A,B] \}.
\]
Given an initial distribution for $\omega_0$, let
$p_A$ be the probability to escape through the left boundary $A$.
For $\omega_0$ chosen from the stationary distribution, 
\[
p_A  = \nu \{ \omega \; ; \; w_{T_{[A,B]}(\omega)} < A \}.
\]
We assume 
\begin{align}\label{e:ABG}
|A|, B &> G,
\end{align}
where $G$ is given in \eqref{e:deltabound}.

\begin{proposition}\label{p:a>0AB}
With probability one, $T_{[A,B]}<\infty$.

There is $\alpha_0>0$ so that for $0< \alpha < \alpha_0$, the following holds.  There exists $c_1,c_2$, such that
\[c_1 \alpha e^{(2/V^-) \alpha A}   \le p_A \le c_2 \alpha e^{(2/V^+)\alpha A}\] 
if $B$ is bounded and $|A|$ is large.
 There exists $C$, such that
   \[ 
     \mathbb{E} (T \vert w_{T_{[A,B]}} < A ) \le C \frac{1}{\alpha} e^{(-2/V^-) \alpha A}.\]
\end{proposition}

\begin{remark}\label{r:a>0AB}
If we replace \eqref{e:varbound} by a one-sided bound 
 \begin{align}\label{e:varbound+}  
\sum_{j=1}^K  \pi_{ij}  \zeta (i,j)^2 &\le V^+,
\end{align}
we retain the estimate
\[
p_A \le c_2 \alpha e^{(2/V^+)\alpha A}.
\]
For this we only need to assume $|A| > G$ and not $B>G$ as in \eqref{e:ABG} since decreasing $B$ also decreases $p_A$.  
\end{remark}

\begin{proof}
For notational convenience we will write $T$ for $T_{[A,B]}$ in this proof.
That the escape time is finite for almost all $\omega\in\Sigma_\mathcal{A}$ was established in Lemma~\ref{l:Tfinite}
as a consequence of \eqref{e:xM>L}.

For $r \ne 0$ and recalling \eqref{e:vn} consider 
\[
z_{n} = e^{r v_n}.
\]
Denote $v_{n+1}=v_n+\zeta_n+\alpha$ and
calculate
\[
\mathbb{E} (z_{n+1} \vert \mathcal{F}_n) = \mathbb{E} (e^{r (\zeta_n + \alpha)}  \vert \mathcal{F}_n)  z_n. 
\]
We wish to find $r$ so that $z_n$ is a submartingale for which
\[
\mathbb{E} (z_{n+1} \vert \mathcal{F}_n) \ge z_n. 
\]
For this we need
\[
\mathbb{E} (e^{r (\zeta_n + \alpha)}  \vert \mathcal{F}_n) \ge 1.
\]
That is,
\begin{align*}
\sum_i  \pi_{\omega_{n-1}i} e^{r (\zeta (\omega_{n-1},i) + \alpha)} &\ge 1.
\end{align*}
For $r$ small we may develop the exponential function in a Taylor series (this is justified since $\zeta(i,j)$ are bounded by \eqref{e:stepbound}).
Doing so yields 
\[
\sum_i  \pi_{\omega_{n-1}i}  \left(  1 + r (\zeta  (\omega_{n-1},i) + \alpha) + \frac{1}{2} r^2  (\zeta  (\omega_{n-1},i) + \alpha)^2    + O (r^3) \right) \ge 1
\]
which, using \eqref{e:pzeta=0}, is equivalent to
\[
r \alpha + \frac{1}{2} r^2 V + O(r^3) \ge 0
\]
with variance \[V = V (\omega_{n-1}) = \sum_i  \pi_{\omega_{n-1}i}  (\zeta  (\omega_{n-1},i) + \alpha)^2.\]
By \eqref{e:varbound},
this is solved for an $r^-<0$ with 
\[
r^- \le \frac{-2}{V^-} \alpha 
\] 
for $\alpha$ small.

Now we can use this to calculate probabilities to escape through the left or right boundary of $[A,B]$.
Doob's optional stopping theorem yields
\begin{align}\label{e:zTz0}
\mathbb{E} (z_T) \ge \mathbb{E} (z_0) &= 1.
\end{align}
Recall that $D$ bounds the stepsize, see \eqref{e:stepbound}.
If $w_T < A$ then $w_T \in (A-D,A)$ and thus by \eqref{e:vnwn} and \eqref{e:deltabound} we have $v_T \in (A-G-D,A+G)$. Likewise 
if $w_T > B$ then $w_T \in (B,B+D)$ and thus $v_T \in (B-G,B+G+D)$.
From \eqref{e:zTz0}  we get $p_A e^{r^- (A-c_A)}  + (1-p_A) e^{r^-(B+c_B)}  \ge 1$ for some $c_A, c_B \in [-G,G+D]$.
Note that, since we assumed $|A|,B > G$, we have $A-c_A <0$ and $B+c_B >0$.
So
\[
p_A e^{r^-A} e^{-r^- c_A} + (1-p_A) e^{r^-B}e^{r^- c_B}   \ge 1.
\]
and hence
\[
p_A (e^{r^-A} e^{-r^- c_A} - e^{r^-B}e^{r^- c_B}  ) \ge 1 - e^{r^-B}e^{r^- c_B}.
\]
Therefore,
\begin{align}\label{e:pA}
p_A \ge \frac{1 - e^{r^-B}e^{r^- c_B} }{e^{r^-A}e^{-r^- c_A}  -e^{r^-B}e^{r^- c_B}}.
\end{align}
 For $B$ fixed and $|A|$ large we find 
 \[
 p_A \ge - c r^- e^{-r^- A}
 \] 
for some constant $c$.

A similar calculation gives that $e^{rv_n}$ is a supermartingale for an $r^+ <0$ with 
\[
r^+ \ge \frac{-2}{V^+} \alpha 
\] 
for $\alpha$ small.
Using the supermartingale $e^{r^+ v_n}$ instead one finds instead of \eqref{e:pA} the similar inequality 
\begin{align*}
p_A \le \frac{1 - e^{r^+B}e^{r^+ c_B} }{e^{r^+  A}e^{-r^+ c_A}  -e^{r^+ B}e^{r^+ c_B}}.
\end{align*}
For $B$ fixed and $|A|$ large we find 
\[
p_A \le - c r^+ e^{-r^+ A}
\] 
for some constant $c$.

Having estimated the probabilities of escape through the left and right boundary, we can now estimate escape times. 
Let $z_n = v_n - n \alpha$. Thus $z_{n+1} = z_n + \zeta_n$ and $z_n$ is a martingale
by \eqref{e:pzeta=0}:
\begin{align*}
\mathbb{E}(z_{n+1} \vert \mathcal{F}_n) &= v_n - n \alpha + \mathbb{E}(\zeta_n \vert \mathcal{F}_n)
\\
&= z_n.
\end{align*}

By Doob's optional stopping theorem, 
\[
\mathbb{E} ({z}_T) = \mathbb{E} (z_0) = 0.
\]
This gives $p_A (A - c_A) + (1-p_A) (B+c_B) - \mathbb{E} (T) \alpha = 0$, so that
\begin{align}\label{e:pAC1}
\mathbb{E} (T)  = \frac{1}{\alpha} \left( p_A (A - c_A) + (1-p_A) (B+c_B) \right).
\end{align}
For $B$ fixed and $|A|$ large we get 
$\mathbb{E} (T) \le c$ for some constant $c$.

For the expected time to escape conditioned by escaping through the left boundary, 
\[\mathbb{E} (T \vert v_{T} < A ) \le \mathbb{E}(T) /p_A \]  
implies the given bound.
\end{proof}

We continue with a result on escape from an unbounded interval $(-\infty,B]$.
Take $w_0 = 0$, assume 
\begin{align}\label{e:BG}
B &> G,
\end{align}
and define 
\[
T_B = \min \{ n>0 \; ; \; w_n \not \in  (-\infty,B]\}.
\] 
The following proposition uses only an upper bound \eqref{e:varbound+} as in Remark~\ref{r:a>0AB}.

\begin{proposition}\label{p:a>0}
With probability one, $T_B < \infty$.
For some $c_1,c_2$, 
\begin{align*}
\frac{c_1}{\alpha} \le \mathbb{E} (T_B) \le \frac{c_2}{\alpha}. 
\end{align*}
\end{proposition}

\begin{proof}
Again we write $T$ as shorthand for $T_B$ inside this proof.
%
Let $u_n = v_n - n \alpha$. 
As noted earlier, 
$u_n$ is a martingale. 
Making $|A|$ bigger in Proposition~\ref{p:a>0AB} gives an increasing subset of $\Sigma_{\mathcal{A}}$ that leads to an escape through $B$.
In particular, $p_A$ goes to $0$ as $|A|\to \infty$.
If we let $|A| \to \infty$ in Proposition~\ref{p:a>0AB} we get $T<\infty$ almost everywhere.
By monotone convergence and \eqref{e:pAC1} we get $\mathbb{E} (T) < \infty$.
By Doob's optional stopping theorem, 
\[
\mathbb{E} (u_T) = \mathbb{E} (u_0).
\]
We find  $B + c - \mathbb{E} (T) \alpha = v_0$ for some $c \in [-G,D+G]$.
The proposition follows (noting $v_0 - B - c \ne 0$).
\end{proof} 

\subsection{Zero drift \texorpdfstring{$\alpha=0$}{alpha=0}}\label{ss:=0}

Again we first treat escape from a compact interval $[A,B]$ with $A<0$, $B>0$, $v_0=0$. Notation is as in the previous part treating $\alpha >0$.
We assume \eqref{e:ABG} to hold.

\begin{proposition}
With probability one, $T_{[A,B]}<\infty$.

There is $\alpha_0>0$ so that for $0< \alpha < \alpha_0$, the following holds.  
There exists $c_1,c_2$, 
such that
\[
c_1/|A|\le p_A \le c_2/|A|
\] 
if $B$ is bounded and $|A|$ is large.

Then \[\mathbb{E} (T\vert v_{T_{[A,B]}} < A) \le c A^2\] for some constant $c>0$, if $|A|$ is large.
\end{proposition}

\begin{proof}
As earlier we write $T$ for $T_{[A,B]}$ inside this proof.
Note that with $\alpha = 0$, $v_n$ is a martingale. Doob's optional stopping theorem gives
$\mathbb{E} (v_T) = \mathbb{E} (v_0) = 0$. Therefore
$p_A A + (1-p_A) B$ is constant and the bounds on $p_A$ follow.
%

To estimate the stopping time conditional to an escape through the left boundary $A$, we introduce 
 $u_n = v_n^2 - n V^-$. 
Then
\begin{align*}
\mathbb{E} ( u_{n+1} \vert \mathcal{F}_n) &=   \mathbb{E} ( (v_n+\zeta)^2 - (n+1) V^-  \vert \mathcal{F}_n)
\\
&\ge v_n^2 -n V^-
\\
&= u_n,
\end{align*}
so that $u_n$ is a submartingale. Likewise a supermartingale can be created, replacing $V^-$ with $V^+$. 
Then $\mathbb{E} (u_T) \ge \mathbb{E} ({u}_0) = 0$ by Doob's optional stopping theorem, giving
\[ p_A A^2 + (1-p_A) B^2 - \mathbb{E}(T) V^- \ge 0.\]
We get $\mathbb{E} (T) \le c A$ for a constant $c>0$. Likewise, calculating with the supermartingale, $\mathbb{E} (T) \ge c A$ for a constant $c>0$. 
So $\mathbb{E} (T \vert v_{T} < A )  \le c A^2$ for some constant $c>0$.  
\end{proof}

The next proposition treats escape from $(-\infty,B]$ for $v_0 <B$.

\begin{proposition}
With probability one, $T_B < \infty$. The expected escape time is infinite:
\[\mathbb{E} (T_B) = \infty.\] 
\end{proposition}

\begin{proof}
This follows from Proposition~\ref{p:a>0}, letting $\alpha$ go to zero. 
\end{proof}

\subsection{Negative drift \texorpdfstring{$\alpha<0$}{alpha<0}}\label{ss:<0}

For compact intervals this case is similar to the case $\alpha>0$.
We just consider the probability of escape from an unbounded interval $(-\infty,B]$. We assume $w_0=0$ and \eqref{e:BG}.
Write $p_B$ for the probability to never escape. For $\omega_0$ from the stationary distribution, 
$p_B = \nu (\{ \omega  \; ; \; v_n < B \textrm{ for all } n \in \mathbb{N}    \})$.

%

\begin{proposition}\label{p:escapeprob}
There is $\alpha_0<0$ so that for $\alpha_0<\alpha<0$,
\[
 c |\alpha| \le p_B \le C |\alpha|
\]
for some positive constants $c<C$.
\end{proposition}


\begin{proof}
Consider escape from an interval $[A,B]$ with $B$ fixed and $|A|$ large.
As in the proof of Proposition~\ref{p:a>0AB} we get for the probability $p_A$ of escape through $A$,
\[
p_A e^{r^-(A - c_A)} + (1-p_A)e^{r^-(B+ c_B)} \ge 1,
\]
where $r^- \ge (-2/V^-)\alpha$ for $|\alpha|$ small.
From this we obtain
\begin{align*}
p_A \left( e^{r^- A}e^{-r^- c_A}  - e^{r^- B}e^{r^- c_B}\right) &\ge 1 - e^{r^- B}e^{r^- c_B},
\end{align*}
and thus
\begin{align*}
p_A &\le \frac{e^{r^- B}e^{r^- c_B} - 1 }{e^{r^- B}e^{r^- c_B} - e^{r^- A}e^{-r^- c_A}},
\end{align*}
compare \eqref{e:pA}.
%
%
With $r^- > 0$,  $e^{r^- A}e^{-r^- c_A}$ goes to zero as $A \to -\infty$.
Hence, from
\begin{align}\label{e:C5est}
p_A &\le \frac{1 - e^{ - r^- B}e^{r^- c_B} }{ 1 - e^{-r^- B}e^{-r^- c_B}e^{r^- A}e^{-r^- c_A}},
\end{align}
we find $p_A \le C |\alpha|$ for a constant $C$ independent of $A$.
%
When taking $|A|$ larger, 
for increasing subsets of $\Sigma_\mathcal{A}$, orbits escape through $B$.   
Letting $A$ go to $-\infty$ therefore proves the upper bound on $p_B$.

The lower bound is proved similarly using a lower bound for $p_A$ as in the proof of Proposition~\ref{p:a>0AB}.
\end{proof}

\section{Doubling maps}\label{a:fordoubling}

The previous appendices considered general subshifts with a primitive adjacency matrix.
The results on the stopping times in Appendix~\ref{a:stop} rely on bounds for the solution of the Poisson equation. 
Here we specialize to subshifts corresponding to Markov partitions for   doubling maps, and 
derive the bounds for this case.


Recall the setting from Section~\ref{s:approxbystep}: given is the subshift $\sigma$ on $\Sigma_{\mathcal{A}_N} \subset \{ 1,\ldots,K\}^\mathbb{N}$, $K=2^N$,
with adjacency matrix  $\mathcal{A}_N = (a_{ij})^{K}_{i,j=1}$ satisfying $a_{ij} = 1$ precisely if $j = 2i-1$ or $j=2i$ modulo $K$.
We start with a remark on subshifts on $\Sigma_{\mathcal{A}_N}$ for different values of $N$.
There is a natural topological conjugacy between the shifts $\sigma$ on $\Sigma_{\mathcal{A}_{N_0}}$ and $\Sigma_{\mathcal{A}_{N_1}}$ for different $N_0,N_1$, 
given by a homeomorphism $\Theta_{N_0,N_1}$,
\begin{align*}
 \sigma\vert_{\Sigma_{\mathcal{A}_{N_0}}} = (\Theta_{N_0,N_1})^{-1} \circ \sigma\vert_{\Sigma_{\mathcal{A}_{N_1}}} \circ \Theta_{N_0,N_1}.
\end{align*}
Given the product measures $\nu_{N_0}, \nu_{N_1}$ on the spaces $\Sigma_{\mathcal{A}_{N_0}}$ and $\Sigma_{\mathcal{A}_{N_1}}$, we have
\[ 
(\Theta_{N_0,N_1})_*  \nu_{N_0} = \nu_{N_1}.
\]
The topological conjugacy $\Theta_{N_0,N_1}$ between the shift operators on these spaces thus provides a measurable isomorphism.
For $N_1 > N_0$, $\sigma^{N_1-N_0}$ maps a cylinder of rank one in $\Sigma_{\mathcal{A}_{N_0}}$ one-to-one to a cylinder of rank one in 
$\Sigma_{\mathcal{A}_{N_1}}$.

Consider a Markov random walk 
\[
 x_{n+1} = x_n + \xi_N (\sigma^n \omega) 
\]
with $\omega\in \Sigma_{\mathcal{A}_N}$.
In our setting the steps $\xi_N$ come from discretizing a smooth function $\xi: \I \to \mathbb{R}$.
Assume 
\[ 
P_\nu \xi_N = 0
\] 
(see Appendix~\ref{a:subshift}).
Lemma~\ref{l:poisson} and Lemma~\ref{l:martingale} give the existence of a function $\Delta_N: \Sigma_{\mathcal{A}_N} \to \mathbb{R}$ so that 
\begin{align*}
u_n &= x_n - n P_\nu \xi_N  - \Delta_N(\sigma^{n-1}\omega) + \Delta_N(\omega)
\end{align*}
is a martingale with respect to $\mathcal{F}_n$.
We recall
\begin{align*}
 u_{n+1} &= u_n + \xi_N (\sigma^n \omega) - \Delta_N (\sigma^{n}\omega) + \Delta_N (\sigma^{n-1} \omega)
 \\
 &= u_n + \zeta_N (\omega_{n-1},\omega_n).
\end{align*}

The following example treats the symmetric random walk in this setup. The steps are $\pm 1$ and do not originate from discretizing a smooth function $\xi$.

\begin{example}[The symmetric random walk]\label{ex:srw}
The symmetric random walk with steps $\pm 1$ is given by the stochastic matrix $\Pi_1$ and vector $\xi_1$,
\[ \Pi_1 = \left(\begin{array}{cc} 1/2 & 1/2 \\ 1/2 & 1/2 \end{array} \right), \qquad \xi_1 =  \left( \begin{array}{c} 1 \\ -1 \end{array}\right).\] 
Here, of course, \[\Delta_1 = \left( \begin{array}{c} 0 \\ 0 \end{array}\right).\]   
%
A refined partition is made from $2$-tuples of consecutive symbols $(\omega_0\omega_1)$.
For this refined Markov partition on four symbols $1,2,3,4$ we get to consider 
\[ 
\Pi_2 =  \frac{1}{2} \left( \begin{array}{cccc} 1 & 1 & 0 & 0 \\  0 & 0 & 1 & 1   \\ 1 & 1 & 0 & 0 \\  0 & 0 & 1 & 1    \end{array} \right), 
\qquad \xi_2 = \left( \begin{array}{c} 1 \\ 1 \\ -1 \\ -1 \end{array}\right).
\]
Here 
$(\Pi_2 - id) \Delta_2 = \Pi_2 \xi_2$ 
is solved by
\[ 
\Delta_2 = \left( \begin{array}{c} -1 \\ 1 \\ -1 \\ 1 \end{array}\right).
\]
%

Continuing the previous examples one consider $N$-tuples of consecutive symbols $(\omega_0\cdots\omega_{N-1})$. 
As there are $2^N$ such $N$-tuples,
one obtains $2^N\times 2^N$-matrices $\Pi_N$ and $2^N$ dimensional vectors $\xi_N$.
That is, for $\Pi_N = (\pi_{ij})_{1\le i,j\le 2^N}$ and giving the $N$-tuples the lexicographical order, we find
\[
 \pi_{ij} = \left\{ \begin{array}{ll} 1/2, &  j \in \{ 2i-1, 2i \} \mod 2^N,  \\  0, & \textrm{otherwise}. \end{array}\right.
\]
For $\xi_N$, for the first $2^{N-1}$ indices the value is $1$, for the others the value is $-1$.
One can characterize $\Delta_N$ as
\[
\Delta_N(i) = \xi_N(i) + |\xi_N(i)| \left( \sharp \{0\le j\le N-1 \; ; \; \omega_j = 2 \} - \sharp  \{0\le j\le N-1 \; ; \; \omega_j = 1 \}\right)
\]
if
$i$ corresponds to the $N$-tuple $(\omega_0\cdots\omega_{N-1})$.
Figure~\ref{f:n=10bernoullidot} illustrates graphically the solution $\Delta_N$ for $N=10$.
\begin{figure}[phtb]

\begin{center}
\includegraphics[width=14cm]{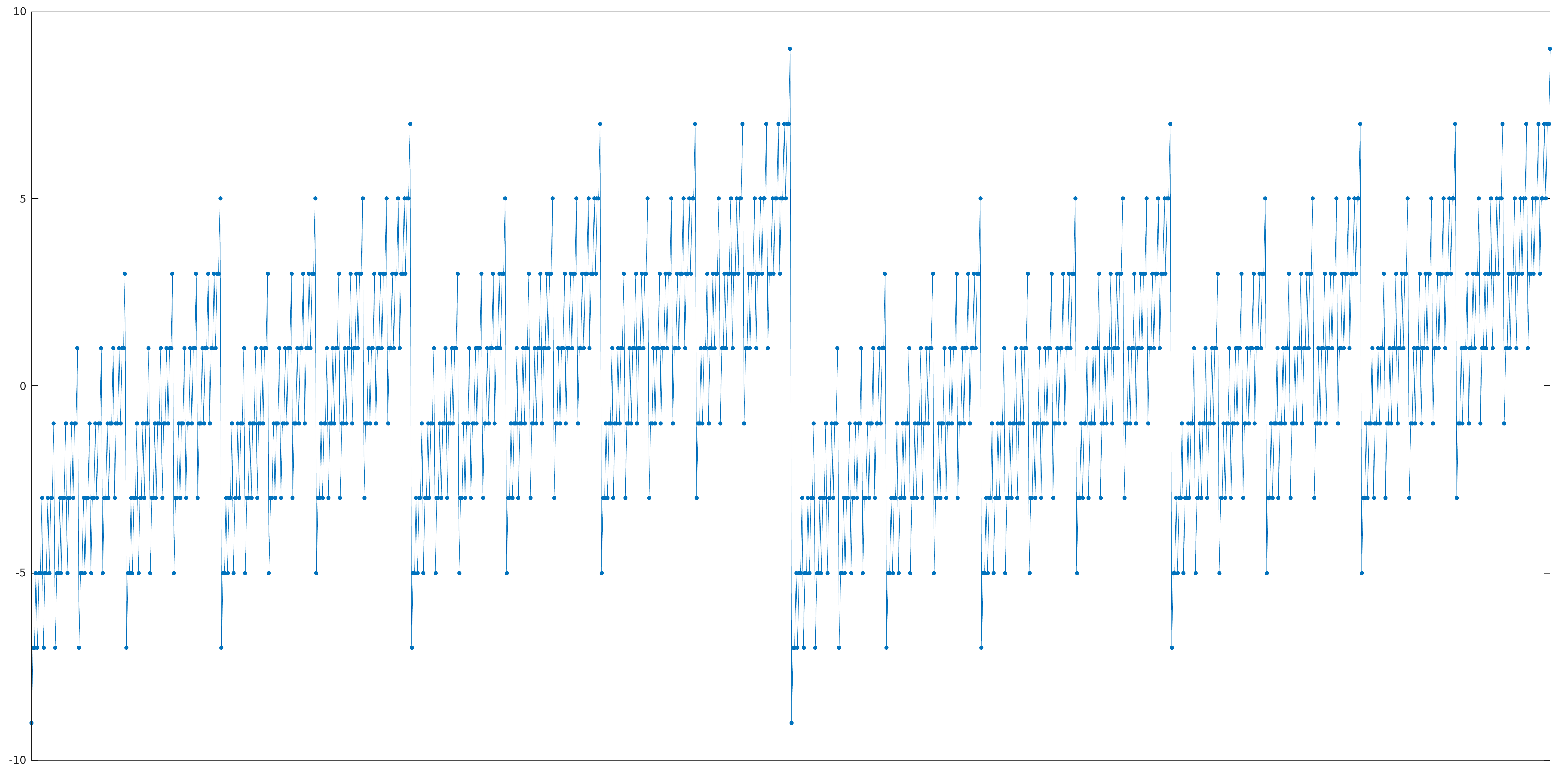}
\caption{This figure illustrates the solution of the Poisson equation for a fine Markov partition coming from the symmetric random walk in Example ~\ref{ex:srw}.
Depicted is the solution $\Delta_N$ for the Poisson equation for $N=10$, i.e. for a Markov partition with $1024$ elements.  The horizontal axis contains the indices $1$ to $1024$, the vertical axis the corresponding value of $\Delta_N$. 
The values are connected by line pieces.
\label{f:n=10bernoullidot}}
\end{center}

\end{figure}
\end{example}



%
%

The following two lemmas provide the bounds for solutions of Poisson equations needed in Appendix~\ref{a:stop}
(see \eqref{e:deltabound}, \eqref{e:stepbound} and \eqref{e:varbound}), discussing dependence on $N$.
The first lemma assumes a smooth function $\xi$ and does not require $\xi$ to be strictly monotone.

\begin{lemma}\label{l:boundDV}
Assume that $\xi: [0,1] \to \mathbb{R}$ is a smooth function with $\int_0^1 \xi(y) \, dy = 0$.
There are uniform (i.e. independent of $N$) positive constants $C,D,V^+$, with
\begin{align*}
|\Delta_N| \le C N,
\end{align*}
for all admissible $i,j$,
\begin{align*}
 \vert \zeta_N (i,j) \vert \le D
\end{align*}
and
\begin{align}\label{e:varbound2}  
\sum_{j=1}^K  \pi_{ij}  \zeta_N (i,j)^2 &\le V^+.
\end{align}
\end{lemma}

\begin{proof}
We know
\[
(\Pi_N - \mathrm{id}) \Delta_N  = \Pi_N \xi_N.
\]
Since here
\[
\Pi_N^{N} \xi_N = 0,
\] 
we find 
\begin{align}\label{e:DeltaN}
\Delta_N &= (\Pi_N + \cdots +\Pi_N^{N-1})  \xi_N.
\end{align}
Powers of $\Pi_N$ are given by
\begin{align}\label{e:PIMN}
 \left(\Pi_N^M\right)_{ij} &= \left\{  \begin{array}{ll} 
                                               1/2^M, & j \in [2^M (i-1)+1,2^M] \mod K, \\
                                               0, & \textrm{otherwise}.
                                      \end{array}
\right.
\end{align}
Since $\xi_N$ is bounded we get that $|\Pi_N^j\xi_N|$ is bounded and thus 
the first bound stating $|\Delta_N| \le CN$ for some $C>0$, follows.

Since $\xi$ is smooth we get  
\begin{align}\label{e:smooth}
\vert (\Pi_N^j  \xi_N)_k - (\Pi_N^j \xi_N)_{k+1} \vert \le C / 2^{N-j}
\end{align}
for some $C>0$.
Combining \eqref{e:DeltaN}, \eqref{e:PIMN} and \eqref{e:smooth} proves a bound
\begin{align}\label{e:D-D}
|\Delta_N (2i) - \Delta_N (2i-1)| &\le C
\end{align}
for some $C>0$. 
Note that this bound holds uniformly in $N$.
Lemma~\ref{l:poisson}  yields  (with indices modulo $K$)
\begin{align*}
\zeta_N (i,2i-1) + \zeta_N(i,2i) &= 
\xi_N (2i-1) - \Delta_N (2i-1) + \Delta_N (i) 
\\ &\;\;\;\; + \xi_N(2i) - \Delta_N(2i) + \Delta_N (i) 
\\ 
&= 0.
\end{align*}
We conclude that $|\zeta_N (i,j)|$ (where $j=2i-1$ or $j=2i$ modulo $K$) is uniformly bounded:
\[
 |\zeta_N (i,j)| \le \max_{1\le k \le K} |\xi_N (k)| + C/2,
\]
with $C$ coming from \eqref{e:D-D}.
The bound \eqref{e:varbound2} follows. 
\end{proof}

Figure~\ref{f:n=10linearupdot} shows the solution of the Poisson equation for a discretization of $\xi=-1+2y$, the function that features in Figures~\ref{f:-1+2y} and \ref{f:-1+2ypert}.
\begin{figure}[phtb]

\begin{center}
\includegraphics[width=14cm]{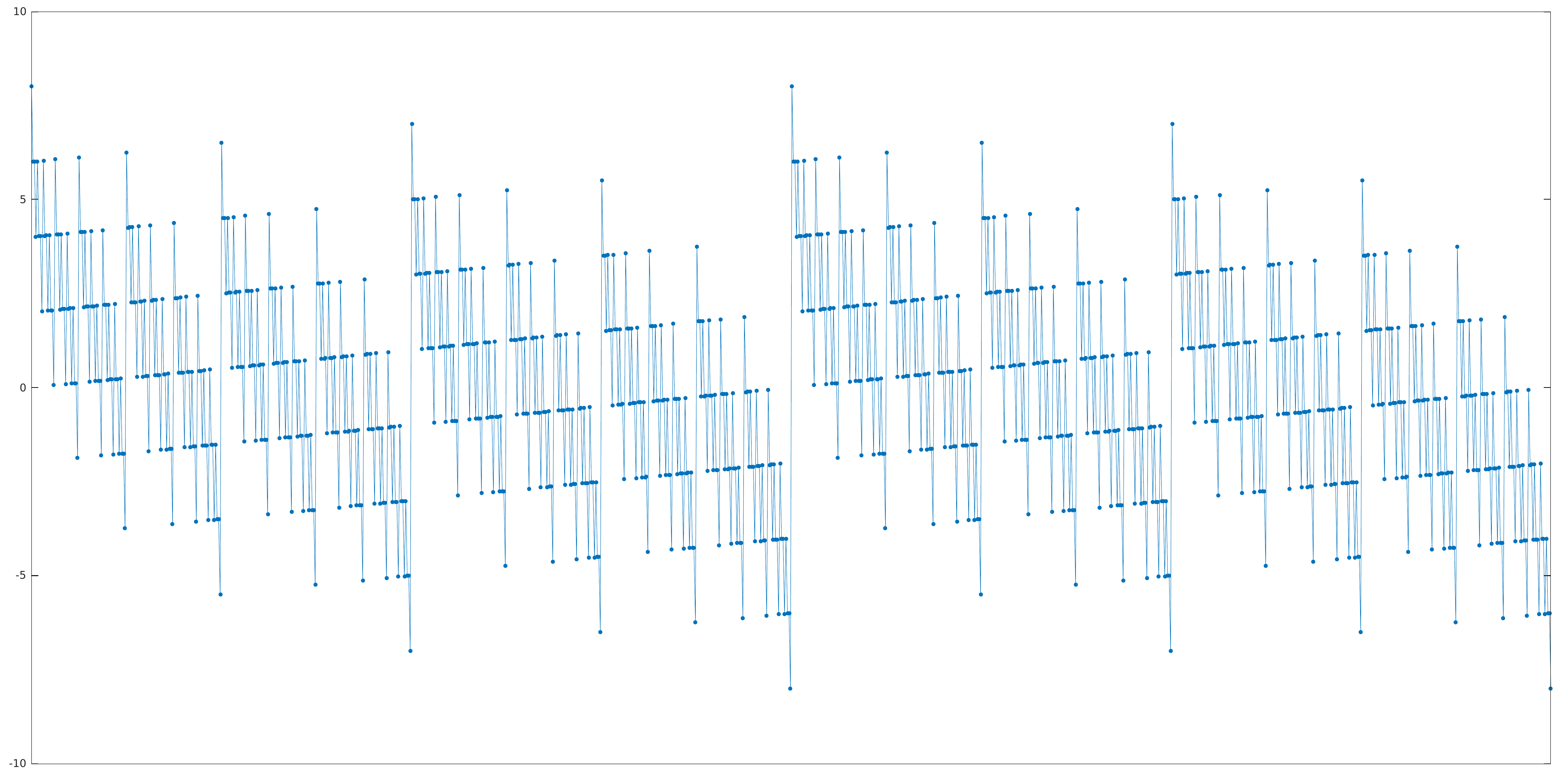}
\caption{This figure illustrates the solution of the Poisson equation when discretizing the function $\xi(y) = -1+2y$ using a fine Markov partition.
Depicted is the solution $\Delta_N$ for the Poisson equation for $N=10$, i.e. for a Markov partition with $1024$ elements.  The horizontal axis contains the indices $1$ to $1024$, the vertical axis the corresponding value of $\Delta_N$. 
The values are connected by line pieces.
\label{f:n=10linearupdot}
}
\end{center}

\end{figure}
The material in Appendix~\ref{a:stop} assumes for some results, in addition to the bounds from Lemma~\ref{l:boundDV}, a positive lower bound 
on $\sum_{j=1}^K  \pi_{ij}  \zeta_N (i,j)^2$ (expressed by \eqref{e:varbound}).
See also Lemma~\ref{l:boundDViter} for a corresponding statement on skew product systems over stronger expanding maps $E_m$. 

\begin{lemma}\label{l:boundV-}
Assume that $\xi: [0,1] \to \mathbb{R}$ is a smooth strictly monotone function with $\int_0^1 \xi(y) \, dy = 0$.
Then there is a positive constant $V^-$ so that, for all $N$ larger than some $N_0$,
\begin{align*}
V^- &\le  
\sum_{j=1}^K  \pi_{ij}  \zeta_N (i,j)^2.
\end{align*}
\end{lemma}

\begin{proof}
Assume for definiteness that $\xi$ is strictly increasing.
The proof of Lemma~\ref{l:boundDV} provided a bound
\[
\Delta_N (2i) - \Delta_N (2i-1) \le C
\] 
for some $C>0$. 
Recall
\begin{align*}
\Delta_N &= (\Pi_N + \cdots +\Pi_N^{N-1})  \xi_N
\end{align*}
and 
\begin{align*}
 \left(\Pi_N^M\right)_{ij} &= \left\{  \begin{array}{ll} 
                                               1/2^M, & j \in [2^M (i-1)+1,2^M] \mod K, \\
                                               0, & \textrm{otherwise}.
                                      \end{array}
\right.
\end{align*}
From these identities and the fact that 
$\xi$ is strictly increasing, we find
\begin{align}\label{e:positive}
0 &< c \le \Delta_N (2i) - \Delta_N (2i-1)
\end{align}
for some $c>0$. Indeed,  it follows that 
$\Delta_N (2i) - \Delta_N (2i-1)$ is a sum of positive values, where $\Pi_N^{N-1} \xi_N  (2i) -  \Pi_N^{N-1} \xi_N  (2i-1)$ is uniformly bounded from below.
The bounds $c$ therefore holds uniformly in $N$.

We conclude that
\[
 |\zeta_N (i,j)| \ge c/2 - \max_{1\le i\le K} (\xi_N (2i) - \xi_N(2i-1))/2,
\]
with $c$ coming from \eqref{e:positive} (where $j=2i-1$ or $j=2i$ modulo $K$).
For $N$ large enough,  $\xi_N (2i) - \xi_N(2i-1)$ are close and therefore $|\zeta_N (i,j)|$  is uniformly bounded from below.
%
%
\end{proof}

The uniform bounds $D, V^-,V^+$ make that the propositions in Appendix~\ref{a:stop} can be applied uniformly in $N$. 
The bound $|\Delta_N| \le C N$ in Lemma~\ref{l:boundDV} implies conditions on the size of the considered intervals, see e.g. 
\eqref{e:ABG} for Proposition~\ref{p:a>0AB}.

\section{General displacement functions}

In this section we show how to prove Theorem~\ref{t:birkhoffT}, with the methods used for Theorem~\ref{t:birkhoff00}. 
For an integer $m \ge 2$, consider skew product maps
\[
 (y,x) \mapsto (E_m (y) , g_y (x))
\]
from $\mathcal{S}^m_{r_0}$ on $\I \times \mathbb{R}$.

The following corollary substitutes Lemma~\ref{l:left}, that was formulated for monotone functions on $\I$. 
Notation is copied from there.
Suppose that $\xi$ is a smooth function on $\I$ with $\int_{\I} \xi (x) \, dx = 0$.

\begin{lemma}
For $m$ large enough, the following property holds.
For $r_0$ small enough, and any $L>0$, there is $\omega$ from a set of positive probability so that for $x \in \mathbb{R}$, there is 
$n\in\mathbb{N}$, $f^n_\omega    (x) < x-L$ 
($f^n_\omega (x) > x+L$) and
$f^i_\omega (x) < x$ ($f^i_\omega(x) > x$) for all $0<i\le n$.
\end{lemma}

\begin{proof}
Since $\xi$ is not identically zero and $\int_{\mathbb{T}} \xi (y) \, dy = 0$, we have two open intervals $A$ and $B$ in the circle such that 
$\xi_{|A} < 0$ and $\xi_{|B} > 0$. 
We know that $E_m$ has a unique fixed point in each interval of the form $I_{m, k} = [\frac{k}{m}, \frac{k+1}{m}], k=0, \dots, m-1$. 
It follows that, for $m$ large enough, $A$ and $B$ each contain a fixed point of $E_m$. 
So let us consider such $m$ and denote the fixed point of $E_m$ in $A$ by $y_A$ and the one in $B$ by $y_B$. 
We have $\xi(y_A) < 0 < \xi(y_B)$. 

For each $n \in \mathbb{N}$ we have that for $y$ close enough to $y_A$,   
\begin{align*}
    x_{n+1} &= x_0 + \xi(y) + \xi(E_m y) + \dots + \xi(E^{n-1}_m y) + \xi(E^n_m y)
\end{align*}
with $\xi(E^j_m y)<0$ for each $0\le j \le n$.
So
\[
x_{n+1} < x_{n}< \dots < x_{0}.
\]
A similar statement applies to points near $y_B$.
Taking $n$ large enough, and thus taking smaller neighborhood of the fixed points, gives the lemma for the cocycle $x_{n+1} = x_n + \xi_n$.

Now if we consider the perturbed skew product then the result still follows for $r_0$ small enough so that $\xi^{\max}_{y, n} + Cr_0 < 0$ and $\xi^{\min}_{y, n} - Cr_0 >0$, where
\begin{align*}
    \xi^{\max}_{y, n} = \max_{0\le i\le n} \xi(E^i_m y)
\end{align*}
and
\begin{align*}
    \xi^{\min}_{y, n} = \min_{0\le i\le n} \xi(E^i_m y).
\end{align*}
\end{proof}

Lemma~\ref{l:boundV-} gets replaced by the following lemma.
We keep the notation from the previous section, except that the doubling map is replaced by $E_m$.
The adjacency matrix $\mathcal{A}_N = (a_{ij})_{i,j=1}^K$ with $K = m^N$ is now given by
$a_{ij} = 1$ precisely if $j = mi + k$ modulo $K$, for $-m < k \le 0$.
The stochastic matrix $\Pi_N$ equals $\frac{1}{m} \mathcal{A}_N$.
The following lemma yields lower bounds for $\zeta_N$. Its proof is an adapted version of the 
proof of Lemma~\ref{l:boundV-}.

\begin{lemma}\label{l:boundDViter}
There exists $m>0$ so that the following holds.
There exists $V^->0$ so that for $N$ larger than some $N_0$, 
\[
V^- \le \sum_{j=1}^K\pi_{ij}\zeta_N ({i,j})^2, 
\]
\end{lemma}

\begin{proof}
All the indices in the proof are taken modulo $m^N$. Note that
\[
 \Pi_N = \left( 
 \begin{array}{cccc} 
 \frac{1}{m} \cdots \frac{1}{m} & 0 \cdots 0 & \cdots & 0 \cdots 0 \\
 0 \cdots 0 & \frac{1}{m} \cdots \frac{1}{m} &  0 \cdots 0 & \cdots \\
 \vdots & \ddots & \ddots & \vdots \\
 0 \cdots 0 & \cdots & 0 \cdots 0 & \frac{1}{m} \cdots \frac{1}{m} \\
 \vdots & \vdots & \vdots & \vdots \\
\frac{1}{m} \cdots \frac{1}{m} & 0 \cdots 0 & \cdots & 0 \cdots 0 \\
 0 \cdots 0 & \frac{1}{m} \cdots \frac{1}{m} &  0 \cdots 0 & \cdots \\
 \vdots & \ddots & \ddots & \vdots \\
 0 \cdots 0 & \cdots & 0 \cdots 0 & \frac{1}{m} \cdots \frac{1}{m} \\
 \end{array} 
 \right)
\]
where $ \frac{1}{m} \cdots \frac{1}{m}$ and $0 \cdots 0$ stand for sequences of $m$ identical numbers. 
For fixed $i\in\{1, \dots, m^N\}$ let $I_i =\{m(i-1)+1,\dots, mi\}$, and note
\begin{align*}
    \sum_{j=1}^K \pi_{ij} \zeta_N ({i,j})^2 = \frac{1}{m} \sum_{j\in I_i}\zeta_N ({i,j})^2.
\end{align*}
We will show that 
\begin{align*}
    \max_{j_1,j_2 \in I_i} |\zeta_N(i, j_1) - \zeta_N(i, j_2)| > c
\end{align*}
for some $c>0$, which implies that
\begin{align*}
    \frac{1}{m} \sum_{j\in I_i}\zeta_N ({i,j})^2 > C
\end{align*}
for some $C>0$ (depending on $m$).

We have
\begin{align*}
    |\zeta_N(i, j_1) - \zeta_N(i, j_2)| &= |\xi_N(j_1) - \Delta_N(j_1) - \xi_N(j_2) + \Delta_N(j_2)|\\
    &\ge  
      |\Delta_N(j_1)-\Delta_N(j_2)| - |\xi_N(j_1) - \xi_N(j_2)|.
\end{align*}
As $\xi$ is assumed to be smooth, we have $|\xi'| \le C_2$ for some $C_2 > 0$.
Now each $\xi_N(j)$ is a representative of $\{\xi(x)\}_{x\in P_j}$ where $|P_j| = 1/m^N$. 
Thus for each $i$, the different values of $\xi_N(j), j\in I_i$, are in $m$ consecutive partition elements 
whose union is an interval of length $1/m^{N-1}$. Hence
\begin{align*}
    |\xi_N(j_1) - \xi_N(j_2)| \le \frac{C_2}{m^{N-1}}.
\end{align*}
It follows that
\begin{align}\label{e:difference of zeta's}
  \max_{j_1,j_2 \in I_i}  |\zeta_N(i, j_1) - \zeta_N(i, j_2)|\ge \max_{j_1, j_2\in I_i}|\Delta_N(j_1)-\Delta_N(j_2)| - \frac{C_2}{m^{N-1}}.
\end{align}
Since $\Pi_N^N\xi_N = 0$, we have
\begin{align*}
    \Delta_N = \Pi_N \xi_N + \dots + \Pi_N^{N-1} \xi_N.
\end{align*}
For $k\in \{1, \dots, N-1\}$, let $B_k = \Pi_N^k \xi_N$.
We have
\begin{align*}
    |\Delta_N(j_1) - \Delta_N(j_2)|\ge |B_{N-1}(j_1) - B_{N-1}(j_2)| - \sum_{k=1}^{N-2}|B_k(j_1) - B_k(j_2)|.
\end{align*}
As $\xi \ne 0$ and $\int_{\mathbb{T}} \xi(y) dy = 0$, for $m$ large enough, there exists $C_1>0$ so that
\begin{align*}
    C_1 <\Big \{ - \min_{j\in I_i} B_{N-1}(j), \max_{j\in I_i} B_{N-1}(j)\Big\}.
\end{align*}
Thus
\begin{align}\label{e:difference of Delta's}
    \max_{j_1, j_2\in I_i}|\Delta_N(j_1) - \Delta_N(j_2)|\ge 2 C_1 - \max_{j_1, j_2\in I_i}\sum_{k=1}^{N-2}|B_k(j_1) - B_k(j_2)|.
\end{align}
For each $k$, we have
\begin{align*}
    B_k(j) = \frac{1}{m^k}\sum_{l=m^k(j-1)+1}^{m^k j}\xi_N(l).
\end{align*}
Thus
\begin{align*}
    |B_k(j_1) - B_k(j_2)|&=\frac{1}{m^k}\bigg|\sum_{i=m^k(j_1-1)+1}^{m^k j_1}\xi_N(i) - \sum_{i=m^k(j_2-1)+1}^{m^k j_2}\xi_N(i) \bigg|\\
    &=\frac{1}{m^k}\bigg|\sum_{j=1}^{m^k}\Big(\xi_N(m^k(j_1-1)+j) -\xi_N((m^k(j_2-1)+j)\Big) \bigg|\\
    &\le \frac{1}{m^k}\sum_{j=1}^{m^k}\Big|\xi_N(m^k(j_1-1)+j) -\xi_N((m^k(j_2-1)+j)\Big|.
\end{align*}
The different values of $\xi_N(m^k(j_1-1)+j)$ and $\xi_N(m^k(j_2-1)+j)$ are within a block of length $m^{k+1}$, thus \begin{align*}
    |B_k(j_1) - B_k(j_2)|\le \frac{C_2}{m^{N-k-1}}.
\end{align*}
Hence
\begin{align}\label{e:difference of B's}
    \max_{j_1, j_2\in I_i}\sum_{k=1}^{N-2}|B_k(j_1) - B_k(j_2)|\le \frac{C_2}{m^{N-1}}\sum_{k=1}^{N-2}m^k\le \frac{C_2}{m-1}.
\end{align}

Combining \eqref{e:difference of zeta's}, \eqref{e:difference of Delta's} and \eqref{e:difference of B's} gives
\begin{align*}
    \max_{j_1,j_2 \in I_i} |\zeta_N(i,j_1) - \zeta_N(i,j_2)|\ge 2C_1 - \frac{C_2}{m-1} - \frac{C_2}{m^{N-1}}.
\end{align*}
Therefore for $m$ and $N$ large enough, $\max_{j_1,j_2 \in I_i} |\zeta_N(i,j_1) - \zeta_N(i,j_2)|>c>0$. 
This ends the proof.
\end{proof}

\begin{proof}[Proof of Theorem~\ref{t:birkhoffT}]
Choose $m$ such that Lemma~\ref{l:boundDViter} applies.
Now the arguments to prove Theorem~\ref{t:birkhoff00} apply to conclude Theorem~\ref{t:birkhoffT}. 
\end{proof}


\begin{thebibliography}{10}



\bibitem{aarkea82}
\newblock J. Aaronson, M. Keane,
\newblock The visits to zero of some deterministic random walks,
\newblock {\em Proc. London Math. Soc.} \textbf{44} (1982), 535--553. 






%
%
%



\bibitem{ashastnic98}
\newblock P. Ashwin, P. Aston, M. Nicol,
 \newblock On the unfolding of a blowout bifurcation,   
 \newblock \emph{Physica D} \textbf{111} (1998), 81--95. 

\bibitem{athdai00}
K.B. Athreya, J. Dai,
\newblock Random logistic maps I,
\newblock {\em Journal of Theoretical Probability} \textbf{13} (2000), 595--608.



\bibitem{athsch03}
K.B. Athreya, H.J. Schuh,
\newblock {Random logistic maps II. The critical case,}
\newblock {\em Journal of Theoretical Probability} \textbf{16} (2003),  813--830.



\bibitem{atk76}
\newblock G. Atkinson,
\newblock Recurrence of co-cycles and random walks,
\newblock {\em J. London Math. Soc.} \textbf{13} (1976), 486--488. 

\bibitem{avidoldursar15}
\newblock A. Avila, D.  Dolgopyat, E. Duryev, O. Sarig,
\newblock The visits to zero of a random walk driven by an irrational rotation,
\newblock {\em Israel J. Math.} \textbf{207} (2015), 653--717. 


\bibitem{bonmil08}
\newblock A. Bonifant, J. Milnor,
\newblock Schwarzian derivatives and cylinder maps,
\newblock \emph{Fields Institute Communications} \textbf{53} (2008), 1--21.


\bibitem{boygor97}
\newblock A. Boyarsky, P. G\'{o}ra,
\newblock {\em Laws of chaos. Invariant measures and dynamical systems in one dimension},
\newblock Birkh\"{a}user, 1997.

\bibitem{fiemeltor05}
\newblock M. Field, I.  Melbourne,  A. T\"{o}r\"{o}k,
\newblock Stable ergodicity for smooth compact Lie group extensions of hyperbolic basic sets,
\newblock {\em Ergodic Theory Dynam. Systems} \textbf{25} (2005),  517--551. 


\bibitem{fuhzha00}
\newblock C.-D. Fuh, C.-H. Zhang,
\newblock Poisson equation, moment inequalities and quick convergence for Markov random walks,
\newblock {\em Stochastic Processes Appl.} \textbf{87} (2000), 53--67. 



\bibitem{ghahom17}
\newblock M. Gharaei, A.J. Homburg, 
\newblock Random interval diffeomorphisms, 
\newblock {\em Discrete Contin. Dyn. Syst. Ser. S} \textbf{10} (2017), 241--272.

\bibitem{gou07}
\newblock S. Gou\"{e}zel,
\newblock Statistical properties of a skew product with a curve of neutral points,
\newblock {\em Ergodic Theory Dynam. Systems} \textbf{27} (2007), 123--151. 

\bibitem{gui89}
\newblock Y. Guivarch, 
\newblock Propri\'et\'es ergodiques, en mesure infinie, de certains systemes dynamiques fibr\'es,
\newblock \emph{Ergodic Theory Dynam. Systems} \textbf{9} (1989), 433--453.

\bibitem{heaplaham94} 
\newblock J.F. Heagy, N. Platt, S.M. Hammel,
\newblock Characterization of on-off intermittency,
\newblock \emph{Phys. Rev. E.} \textbf{49} (1994), 1140--1150.



\bibitem{krasza84}
\newblock A. Kr\'{a}mli, D. Sz\'{a}sz, 
\newblock Random walks with internal degrees of freedom. II: First-hitting probabilities,
\newblock \emph{Z. Wahrscheinlichkeitstheor. Verw. Geb.} \textbf{68} (1984),  53--64. 

\bibitem{levperwil09}
\newblock D.A. Levin, Y. Peres, E.L. Wilmer,
\newblock \emph{Markov chains and mixing times},
\newblock Amer. Math. Soc., 2009.

%


\bibitem{meytwe93}
\newblock S.P. Meyn, R.L. Tweedie,
\newblock {\em Markov chains and stochastic stability},
\newblock Springer Verlag, 1993.


\bibitem{morsma14}
\newblock C.G. Moreira, D. Smania,
\newblock Metric stability for random walks (with applications in renormalization theory),
\newblock in: \emph{Frontiers in Complex Dynamics: In Celebration of John Milnor’s 80th Birthday}, 
Princeton University Press, 2014.


\bibitem{nit00}
\newblock V. Ni\c{t}ic\u{a},
\newblock Examples of topologically transitive skew products,
\newblock \emph{Discrete Cont. Dynam. Sys.} \textbf{6} (2000), 351--360.

 \bibitem{nitpol05}  
\newblock V. Ni\c{t}ic\u{a}, M. Pollicott,
\newblock Transitivity of Euclidean extensions of Anosov diffeomorphisms,
\newblock \emph{Ergodic Theory Dynam. Systems} \textbf{25} (2005), 257--269.


\bibitem{ottsom94}
\newblock E. Ott, J.C. Sommerer,
\newblock Blowout bifurcations: the occurrence of riddled basins and on-off intermittency,
\newblock \emph{Physics Letters A}, \textbf{188} (1994), 39--47.



\bibitem{plaspitre93}
\newblock N. Platt, E.A. Spiegel, C. Tresser,
\newblock On-off intermittency: A mechanism for bursting,
\newblock \emph{Phys. Rev. Lett.}, \textbf{70} (1993), 279--282.





\bibitem{sch77}
\newblock K. Schmidt,
\newblock {\em Cocycles of ergodic transformation groups},
 \newblock Lecture Notes in Mathematics, Vol. 1, MacMillan (India) 1977.



\bibitem{shi84}
\newblock A.N. Shiryayev,
\newblock {\em Probability},
\newblock  Springer Verlag, 1984. 




\end{thebibliography}
\end{document}